\documentclass[reqno]{amsart}
\usepackage{hyperref}
\usepackage[a4paper, left=3cm, right=3cm, top=4cm, bottom=4cm]{geometry}
\vspace{9mm}

\begin{document}

\title[\hfilneg \hfil Multi-term time-fractional stochastic control system]{ Optimal controllability for multi-term time-fractional stochastic systems with non-instantaneous impulses }
\author[A. Afreen, A. Raheem  \&  A. Khatoon \hfil \hfilneg]
{A. Afreen$^{*}$, A. Raheem \& A. Khatoon}

\address{Asma Afreen \newline Department of Mathematics,
	Aligarh Muslim University,\newline Aligarh -
	202002, India.} \email{afreen.asma52@gmail.com}

\address{Abdur Raheem \newline
	Department  of Mathematics,
	Aligarh Muslim University,\newline Aligarh -
	202002, India.} \email{araheem.iitk3239@gmail.com}

\address{Areefa Khatoon \newline
	Department  of Mathematics,
	Aligarh Muslim University,\newline Aligarh -
	202002, India.} \email{areefakhatoon@gmail.com}

\renewcommand{\thefootnote}{} \footnote{$^*$ Corresponding author:
	\url{ A. Afreen (afreen.asma52@gmail.com)}}



\begin{abstract}
In the present paper, we study the existence and optimal controllability of a multi-term time-fractional stochastic system with non-instantaneous impulses. Using semigroup theory, stochastic analysis theory, and Krasnoselskii's fixed point theorem, we first establish the existence of a mild solution. Further, we obtain that there exists an optimal state-control pair of the system.  Some examples are given to illustrate the abstract results. 
\end{abstract}

\maketitle \numberwithin{equation}{section}
\newtheorem{theorem}{Theorem}[section]
\newtheorem{lemma}[theorem]{Lemma}
\newtheorem{proposition}[theorem]{Proposition}
\newtheorem{corollary}[theorem]{Corollary}
\newtheorem{remark}[theorem]{Remark}
\newtheorem{definition}[theorem]{Definition}
\newtheorem{example}[theorem]{Example}
\allowdisplaybreaks

\section{\textbf{Introduction}} 

The theory of fractional-order differential systems has received significant attention in recent decades due to remarkable applications in various fields of science and engineering, including medical models, electrical engineering, and so on. Fractional differential and integral operators have proven to be extremely useful due to nonlocal characteristics. Moreover, the fractional differential equations of arbitrary
order can describe the dynamical behavior of real-life phenomena more precisely. The importance of a thorough understanding of fractional calculus has increasingly been recognized. For more information, refer to \cite{a11,a3,i1,k3}.

It is well-known that many real-life phenomena are affected by the sudden change in their state at certain moments, such as heartbeats and blood flow in the human body. These phenomena are discussed as impulsive differential equations. In general, there are two types of impulsive effects. The first type is instantaneous impulsive differential equations, i.e., the duration of these changes is negligible compared to the whole process. The second type is non-instantaneous impulsive differential equations, i.e., the impulsive action starts at fixed points and remains active on a finite time interval. The study of impulsive systems is crucial to analyzing more realistic mathematical models. It has numerous applications in various fields such as drug diffusion in the human body, population dynamics, theoretical physics, mathematical economy, chemical technology, industrial robotics, engineering, etc. For more details, refer to \cite{6,1,2,v2}. 

Many real-world phenomena, such as population growth, stock prices, weather prediction model, and heat conduction in materials with memory, are affected by random influences or noise. So, we have to introduce the presence of randomness into mathematical representations of the given phenomena.
Stochastic differential systems emerge as effective techniques for constructing and examining such occurrences. Besides mathematics, this theory can be extensively applied to a variety of problems in biostatistics, chemistry, economics, mechanics, and finance. For more information, see \cite{AMAAM,h,J,k,l1}.

The general definition of an optimal control problem requires the minimization of a criterion function of the states and control inputs of the system over a set of admissible control functions. R. Dhayal et al. \cite{r1} investigated the solvability and optimal controls of second-order SDE driven by mixed fBm with non-instantaneous impulses. For more informations, refer to \cite{aa,2,r4,t1}.

In recent years, multi-term time-fractional differential systems have played a significant role in various branches of engineering science. In \cite{v5,v6}, the authors investigated a two-term time-fractional differential system in the abstract context, including a concrete example of a fractional diffusion-wave problem. In 2018 \cite{v4}, V. Singh and D.N. Pandey studied a multi-term time-fractional system and gave some new results on controllability. In 2020 \cite{r6}, R. Chaudhary et al. investigated the controllability results of a multi-term time-fractional system with state-dependent delay. However, in 2021 \cite{v3}, V. Singh et al. examined a multi-term time-fractional stochastic system using the Picard iterative scheme.

Motivated by the above works and adapting the ideas used in \cite{r1,v3,v2}, we extend the work of \cite{r1} to a multi-term time-fractional stochastic system. Our aim is to obtain optimal control of the following multi-term time-fractional stochastic system with non-instantaneous impulses in a separable Hilbert space  $\mathcal{Z}$
\begin{eqnarray} \label{1.1}
	\left \{ \begin{array}{lll} ^cD_{t}^{1+\alpha}z(t)+\sum \limits_{\iota=1}^{m}\beta_\iota ^cD_{t}^{\gamma_\iota}z(t)=Az(t)+Eu(t)+g_1\big(t,z(t)\big)+g_2\big(t,z(t)\big)\dfrac{d\upsilon(t)}{dt}, \\\hspace{7.8cm} t \in(e_q,t_{q+1}], ~q=0,1,2,\ldots,r,
		&\\z(t)=\varsigma_q\big(t,z(t_q^-)\big), \quad t \in (t_q,e_q],~ q=0,1,2,\ldots,r,
		&\\z'(t)=\varphi_q\big(t,z(t_q^-)\big), \quad t \in (t_q,e_q],~ q=0,1,2,\ldots,r,
		&\\  z(0)=z_0, \quad z'(0)=z_1,
		
	\end{array}\right.
\end{eqnarray}
where $0<\alpha \leq \gamma_m\leq \cdots \leq \gamma_1< 1,~\beta_{\iota}>0,\iota=1,2,\ldots,m,$ $ ^cD_{t}^{\alpha}$ stands for the Caputo fractional derivative of order $\alpha,$ and $0=e_0=t_0<t_1<e_1<t_2<\ldots<t_r<e_r<t_{r+1}=\ell<\infty,~ J=[0,\ell].$ The state function $z(t) \in \mathcal{Z}$ and $A:\mathcal{D}(A)\subset  \mathcal{Z}\rightarrow  \mathcal{Z}$ be a closed linear operator. Let $\mathcal{W}$ be another real separable Hilbert space. $\upsilon(t)$ be a $\mathcal{W}$-valued Wiener process with $Tr(Q)<\infty,$ $Q\geq 0$ is nuclear covariance operator on a complete probability space $(\Omega, \Upsilon, P),$ where $\Upsilon_t\subset \Upsilon,$ $t \in [0,\ell]$ is a normal filtration. $\Upsilon_t$ is a right continuous increasing family and $\Upsilon_0$ contains all $P$-null sets. Assume that there exists a complete orthonormal system $\{\xi_n\}_{n=0}^{\infty}$ in $\mathcal{W}$, a bounded sequence of non-negative real numbers $\lambda_n$ such that $Q\xi_n=\lambda_n\xi_n, ~n=1,2,\ldots.$ Thus $\upsilon(t)=\sum_{n=1}^{\infty}\sqrt{\lambda_n}\upsilon_n(t)\xi_n$, where $\upsilon_n(t),~n=1,2,\ldots$ are one-dimensional standard Weiner motions mutually independent over  $(\Omega, \Upsilon, P)$. Also, let $L_2^{0}=L_2\big(Q^{1/2}\mathcal{W},\mathcal{Z}\big),$  the space of all Hilbert-Schmidt operators from $Q^{1/2}\mathcal{W}$ to $\mathcal{Z}$ be a separable Hilbert space with 
$$\|\psi\|^2_Q=Tr(\psi Q \psi^*)=\sum\limits_{n=1}^{\infty}\|\sqrt{\lambda_n}\psi \xi_n\|^2,$$ $\psi \in \mathcal{L}(\mathcal{W},\mathcal{Z}),$ the space of bounded linear operators from $\mathcal{W}$ to $\mathcal{Z}.$  The functions $\varsigma_q\big(t,z(t_q^-)\big)$ and $\varphi_q\big(t,z(t_q^-)\big)$ represent non-instantaneous impulses during the intervals $(t_q,e_q],~q=1,2,\ldots,r.$ The control function $u$ takes value in a separable Hilbert space $\mathcal{U}$ and $E \in \mathcal L\mathcal{(U,Z)}$ is a linear operator. The functions $g_1:J\times \mathcal{Z}\rightarrow \mathcal{Z},$ $g_2:J\times \mathcal{Z}\rightarrow L_2^0,$ $\varsigma_q:(t_q,e_q]\times \mathcal{Z}\rightarrow \mathcal{Z},$ and $\varphi:(t_q,e_q]\times \mathcal{Z}\rightarrow \mathcal{Z},~q=1,2,\ldots,r$ are satisfying some suitable conditions. $z_0$ and $z_1$ denote $\Upsilon_0$-measurable $\mathcal{Z}$-valued random variables.

Now we define the space $\mathcal{PC(Z)}$ formed by all $\Upsilon_t$-adapted, $\mathcal{Z}$-valued measurable stochastic processes $\{z(t):t\in[0,\ell]\}$ such that $z$ is continuous at $t\neq t_q,~z(t_q^-)=z(t_q)$ and $z(t_q^+)$ exists for all $q=1,2,\ldots,r$ endowed with the norm  $$\| z \|_{\mathcal{PC}}=\left( \sup\limits_{t \in [0,\ell]}\mathbb{E} \|z(t)\|^p \right)^{\frac{1}{p}}.$$
Then $\big(\mathcal{PC(Z)},\|\cdot\|_{\mathcal{PC}}\big)$ is a Banach space.

The control function $u(\cdot)\in \mathcal L_{\Upsilon}^p([0,\ell],\mathcal{U}),$ where $ \mathcal L_{\Upsilon}^p([0,\ell],\mathcal{U})$ represents the space of all $\mathcal{U}$-valued stochastic processes, which are $\Upsilon_t$-adapted and measurable satisfying $\mathbb{E}\displaystyle\int\limits_{0}^{\ell} \big\|u(t)\big\|^p_{\mathcal{U}}dt<\infty$ and is a Banach space with the norm
$$\|u\|_{L_{\Upsilon}^p}=\Bigg( \mathbb{E}\int\limits_{0}^{\ell} \big\|u(t)\big\|^p_{\mathcal{U}}dt \Bigg)^{\frac{1}{p}}.$$

\noindent Let $U$ be a non-empty closed, bounded and convex subset of ${\mathcal{U}}.$ The admissible control set is given by 
$$U_{ad}=\big\{u\in \mathcal L_{\Upsilon}^p([0,\ell],{\mathcal{U}})~\rvert~ u(t)\in U\mbox{ a.e. } t\in [0,\ell]\big \}.$$ 

The remaining part of the paper is designed as: Section 2 contains some basic definitions, lemmas, and assumptions. In Section 3, we prove the existence of mild solutions for the proposed system. Section 4 contains the study of optimal control. In Section 5, some examples are given to show the effectuality of the results.

\section{\textbf{Preliminaries and Assumptions}}
For $\rho>0,$ define
\begin{eqnarray*} 
	k_{\rho}(t)=	\left \{ \begin{array}{lll} 0, \hspace{1.4cm} t\leq 0,
		&\\  \frac{1}{\Gamma(\rho)}t^{\rho-1}, \quad t>0.		
	\end{array}\right.
\end{eqnarray*}
 Moreover, for $p,q>0$ we have $(k_p*k_q)(t)=k_{p+q}(t),$ where $*$ stands for convolution. 
 
\begin{definition} \cite{a11} Let $g\in \mathcal{C}^n(\mathbb{R}^*, \mathbb{R}).$ Then the Caputo fractional derivative of order $\rho>0$ is given by
	$$^cD^{\rho}g(t)=\int\limits_{0}^{t}k_{n-\rho}(t-s) D^n g(s)ds, \quad n-1<\rho\leq n, ~~n \in \mathbb{N},$$ and $^cD^{0}g(t)=g(t),$ where $D^n$ represents $n^{th}$ order derivative.
\end{definition}

To define a mild solution for (\ref{1.1}), we have the following family of operators.
\begin{definition}\cite{e2}
	Let $\alpha, \beta_\iota$ and $\gamma_\iota$ be the positive real numbers for all $\iota=1,2,\ldots,m.$ Then, the closed linear operator $A$ genertaes a $(\alpha, \gamma_\iota)$-resolvent family if there exists a function $\mathcal{S}_{\alpha, \gamma_\iota}:\mathbb{R}^+\rightarrow \mathcal{L}(\mathcal{Z})$ which is strongly continuous and $\delta>0$ such that $\Big\{\lambda^{1+\alpha}\sum\limits_{\iota=1}^{m}\beta_\iota\lambda^{\gamma_\iota}:\mbox{Re}~ \lambda>\delta\Big\}\subset \varrho(A),$ resolvent set of $A$ and $$\lambda^\alpha{\bigg(\lambda^{1+\alpha}+\sum\limits_{\iota=1}^{m}\beta_\iota\lambda^{\gamma_\iota}-A\bigg)} ^{-1}z=\int\limits_{0}^{\infty}e^{-\lambda t} \mathcal{S}_{\alpha, \gamma_\iota}(t)z dt, \quad z\in \mathcal{Z}, ~\mbox{Re}~ \lambda>\delta.$$
\end{definition}
\begin{theorem}\label{1}\cite{e2}
	Let $0<\alpha \leq \gamma_m\leq \cdots \leq \gamma_1\leq 1$ and $\beta_\iota\geq0$ be given for all $\iota=1,2,\ldots,m$ and assume that $A$ is the generator of a bounded and strongly continuous cosine family $\{C(t)\}_{t\in[0,\ell]}.$ Then $A$ generates a bounded $({\alpha, \gamma_\iota})$-resolvent family $\{\mathcal{S}_{\alpha, \gamma_\iota}(t)\}_{t\geq0}.$
\end{theorem}
\begin{definition} A stochastic process $z \in \mathcal {PC(Z)}$ is said to be a mild solution of (\ref{1.1}) if for every $t \in J,$ $z(t)$ fulfills $z(0)=z_0,~z'(0)=z_1$ and $z(t)=\varsigma_q\big(t,z(t_q^-)\big),~z'(t)=\varphi_q\big(t,z(t_q^-)\big),~ t \in (t_q,e_q],~q=1,2,\ldots,r$ and 
	\begin{eqnarray}\label{2.1}
	z(t)=	\left \{ \begin{array}{lll}
	\mathcal{S}_{\alpha, \gamma_\iota}(t)z_0+\int_{0}^{t}\mathcal{S}_{\alpha, \gamma_\iota}(s)z_1ds+\sum \limits_{\iota=1}^{m}\beta_{\iota}\int\limits_{0}^{t}\frac{(t-s)^{\alpha-\gamma_{\iota}}}{\Gamma(1+\alpha-\gamma_{\iota})}\mathcal{S}_{\alpha, \gamma_\iota}(s)z_0ds\nonumber\\ +\int\limits_{0}^{t} \mathcal{J}_{\alpha, \gamma_\iota}(t-s)Eu(s)ds +\int\limits_{0}^{t}\mathcal{J}_{\alpha, \gamma_\iota}(t-s) g_1\big(s,z(s)\big)ds\nonumber \\+\int\limits_{0}^{t}\mathcal{J}_{\alpha, \gamma_\iota}(t-s)g_2\big(s,z(s)\big)d\upsilon(s),\hspace{4cm} t\in[0,t_1], ~q=0,
	&\\	\mathcal{S}_{\alpha, \gamma_\iota}(t-e_q)\varsigma_q\big(e_q,z(t_q)\big)+\int_{e_q}^{t}\mathcal{S}_{\alpha, \gamma_\iota}(s-e_q)\varphi_q\big(e_q,z(t_q)\big)ds\nonumber\\+\sum \limits_{\iota=1}^{m}\beta_{\iota}\int\limits_{e_q}^{t}\frac{(t-s)^{\alpha-\gamma_{\iota}}}{\Gamma(1+\alpha-\gamma_{\iota})}\mathcal{S}_{\alpha, \gamma_\iota}(s-e_q)\varsigma_q\big(e_q,z(t_q)\big)ds\nonumber\\ +\int\limits_{e_q}^{t} \mathcal{J}_{\alpha, \gamma_\iota}(t-s)Eu(s)ds +\int\limits_{e_q}^{t}\mathcal{J}_{\alpha, \gamma_\iota}(t-s) g_1\big(s,z(s)\big)ds\nonumber \\+\int\limits_{e_q}^{t}\mathcal{J}_{\alpha, \gamma_\iota}(t-s)g_2\big(s,z(s)\big)d\upsilon(s),\hspace{3.9cm} t\in(e_q,t_{q+1}], ~q=1,2,\ldots,r,
		\end{array}\right.
	\end{eqnarray}
	where $\mathcal{J}_{\alpha, \gamma_\iota}(t)=(k_\alpha*\mathcal{S}_{\alpha, \gamma_\iota})(t).$
\end{definition}

\begin{lemma}\label{lm1} \cite{k} Let  $\mathcal{V}$ be ${L_2^0}$-valued predictable process with $p \geq 2$ such that $\mathbb{E}\Bigg(\displaystyle\int_{0}^{\ell} \|\mathcal{V}(s)\|_{L_2^0}^pds\Bigg)< \infty,$ then
\begin{eqnarray*}\mathbb{E}\Bigg(\sup\limits_{s\in[0,t]} \bigg\| \int_{0}^{s}\mathcal{V}(\xi)dv(\xi) \bigg \|^p \Bigg)&\leq& c_p\sup\limits_{s\in[0,t]}\mathbb{E}\Bigg( \bigg\|\int_{0}^{s}  \mathcal{V}(\xi)dv(\xi)\bigg\|^p\Bigg)\\&\leq& C_p\mathbb{E}\Bigg( \int_{0}^{t}  \|\mathcal{V}(\xi)\|^p_{L_2^0}d\xi\Bigg),~t\in[0,\ell],
\end{eqnarray*}
where $c_p=\Big(\frac{p}{p-1}\Big)^p$ and $C_p=\Big(\frac{p}{2}(p-1)\Big)^{\frac{p}{2}}\Big(\frac{p}{p-1}\Big)^{\frac{p^2}{2}}.$
\end{lemma}

Let $S=\sup\limits_{t\in[0,\ell]}\|\mathcal{S}_{\alpha, \gamma_\iota}(t)\|$ and  $\|\mathcal{J}_{\alpha, \gamma_\iota}(t)z\|=\frac{St^\alpha}{\Gamma(1+\alpha)}\|z\|$ for $z \in \mathcal{Z}.$
Consider the following assumptions:
\begin{itemize}
	\item[(C1)] $\mathcal{S}_{\alpha, \gamma_\iota}(t)$ and $\mathcal{J}_{\alpha, \gamma_\iota}(t),~t>0$ are compact.
	\item[(C2)] The functions $g_1:[0,\ell]\times \mathcal{Z}\rightarrow \mathcal{Z}$ and $g_2:[0,\ell]\times \mathcal{Z} \rightarrow {L}_2^0$ satisfy the following conditions:
		\begin{itemize}
		\item[(a)] The functions $g_1(\cdot, z):[0,\ell]\rightarrow \mathcal{Z}$ and $g_2(\cdot,z):[0,\ell]\rightarrow L_2^0$ are strongly measurable for all $z\in\mathcal{Z},$ and the functions $g_1(t,\cdot):\mathcal{Z}\rightarrow \mathcal{Z}$ and $g_2(t,\cdot):\mathcal{Z}\rightarrow L_2^0$ are continuous for each $t\in[0,\ell].$
		\item[(b)] There exist continuous functions $\mu_1,\mu_2\in L^1([0,\ell],\mathbb{R}^+)$ to ensure that for all $(t,z)\in[0,\ell]\times \mathcal{Z}$
\begin{eqnarray*}
 \|g_1(t,z)\|^p\leq {\mu}_1(t)\|z\|^p,\quad 
	\|g_2(t,z)\|^p_{L_2^0}\leq\mu_2(t)\|z\|^p.
	\end{eqnarray*}
\end{itemize}
	\item[(C3)] The functions $\varsigma_q ,~\varphi_q :(t_q,e_q]\times \mathcal{Z}\rightarrow \mathcal{Z}$ are continuous and there exist positive constants $a_q,b_q,$ $\tilde{a}_q$ and $ \tilde{b}_q,$ $q=1,2,\ldots,r$ such that 
		\begin{eqnarray*}
		\big\|\varsigma_q \big(t,y\big)-\varsigma_q \big(t,{z}\big)\big\|^p \leq a_q\|y-z\|^p,\quad 	\| \varsigma_q (t,y)\|^p\leq \tilde{a}_q\big(1+\|y\|^p\big),\\ \big\|\varphi_q\big(t,y\big)-\varphi_q \big(t,{z}\big)\big\|^p\leq b_q\|y-z\|^p,\quad 	\| \varphi_q (t,y)\|^p\leq \tilde{b}_q\big(1+\|y\|^p\big),
		\end{eqnarray*} for all $y,z\in \mathcal{Z}.$
	\item[(C4)] The following inequalities hold:
	\begin{itemize}
		\item[(a)]$\pounds_{\mathcal{F}_1}=\max\limits_{1\leq q\leq r}\bigg\{a_q,~ 3^{p-1}\bigg(S^pa_q+S^p{\ell}b_q+\sum \limits_{\iota=1}^{m}\beta_{\iota}\Big(\frac{S}{\Gamma(1+\alpha-\gamma_{\iota})}\Big)^p\frac{\ell^{(1+\alpha-\gamma_{\iota})p}}{(1+p \alpha-p \gamma_{\iota})}a_q\bigg)\bigg\}<1.$ 
	\item[(b)] 
		$\pounds_{\mathcal{F}}=\max\limits_{1\leq q\leq r}\bigg\{\tilde{a}_q+6^{p-1}S^p\tilde{a}_q+6^{p-1}S^p\ell\tilde{b}_q+6^{p-1}\sum \limits_{\iota=1}^{m}\beta_{\iota}\bigg(\frac{S}{\Gamma(1+\alpha-\gamma_{\iota})}\bigg)^p\frac{1}{(1+p \alpha-p \gamma_{\iota})}\\\times{\ell^{(1+\alpha-\gamma_{\iota})p}}\tilde{a}_q+6^{p-1}\bigg(\frac{S}{\Gamma(1+\alpha)}\bigg)^p\bigg(\frac{p-1}{\alpha p+p-1}\bigg)^p{\ell^{\alpha p+p-1}}\displaystyle\int_{0}^{t}\mu_1(s)ds\\+ 6^{p-1}C_p\bigg(\frac{S}{\Gamma(1+\alpha)}\bigg)^p\bigg(\frac{p-2}{2\alpha p+p-2}\bigg)^{\frac{p-2}{p}}{\ell^{\frac{2\alpha p+p-2}{2}}}\displaystyle\int_{0}^{t}\mu_2(s)ds\bigg\}<1.$
	 \end{itemize}
 \end{itemize}
\section{\textbf{Existence of Mild Solutions}}
\begin{theorem}\label{th1}
If the hypotheses (C1)-(C4) are satisfied, then for each $u\in U_{ad}$, the system (\ref{1.1}) has at least one mild solution on $[0,\ell].$
	
\end{theorem}
\begin{proof} For a constant $h>0,$ we define
	$$\Omega_h=\{z\in \mathcal{PC(Z)}:\|z\|^p_{\mathcal{PC}}\leq h\}.$$ Clearly, $\Omega_h$ is convex, bounded and closed subset of $\mathcal{PC(Z)}.$ Now, define an operator $\mathcal{F}:\Omega_h\rightarrow \mathcal{PC(Z)}$ by 
	\begin{eqnarray*}
		(\mathcal{F} z)(t)=	\left \{ \begin{array}{lll}
		\mathcal{S}_{\alpha, \gamma_\iota}(t)z_0+\displaystyle\int_{0}^{t}\mathcal{S}_{\alpha, \gamma_\iota}(s)z_1ds+\sum \limits_{\iota=1}^{m}\beta_{\iota}\displaystyle\int_{0}^{t}\frac{(t-s)^{\alpha-\gamma_{\iota}}}{\Gamma(1+\alpha-\gamma_{\iota})}\mathcal{S}_{\alpha, \gamma_\iota}(s)z_0ds\nonumber\\ +\displaystyle\int_{0}^{t} \mathcal{J}_{\alpha, \gamma_\iota}(t-s)Eu(s)ds +\displaystyle\int_{0}^{t}\mathcal{J}_{\alpha, \gamma_\iota}(t-s) g_1\big(s,z(s)\big)ds\nonumber \\+\displaystyle\int_{0}^{t}\mathcal{J}_{\alpha, \gamma_\iota}(t-s)g_2\big(s,z(s)\big)d\upsilon(s), \hspace{3.1cm} t\in[0,t_1], ~q=0,&\\ \varsigma_q\big(t,z(t_q^-)\big), \hspace{6.5cm} t\in (t_q,e_q], ~q=1,2,\ldots,r,
		&\\	\mathcal{S}_{\alpha, \gamma_\iota}(t-e_q)\varsigma_q\big(e_q,z(t_q)\big)+\displaystyle\int_{e_q}^{t}\mathcal{S}_{\alpha, \gamma_\iota}(s-e_q)\varphi_q\big(e_q,z(t_q)\big)ds\nonumber\\ +\sum \limits_{\iota=1}^{m}\beta_{\iota}\displaystyle\int_{e_q}^{t}\frac{(t-s)^{\alpha-\gamma_{\iota}}}{\Gamma(1+\alpha-\gamma_{\iota})}\mathcal{S}_{\alpha, \gamma_\iota}(s-e_q)\varsigma_q\big(e_q,z(t_q)\big)ds\nonumber\\+\displaystyle\int_{e_q}^{t} \mathcal{J}_{\alpha, \gamma_\iota}(t-s)Eu(s)ds +\displaystyle\int_{e_q}^{t}\mathcal{J}_{\alpha, \gamma_\iota}(t-s) g_1\big(s,z(s)\big)ds\nonumber \\+\displaystyle\int_{e_q}^{t}\mathcal{J}_{\alpha, \gamma_\iota}(t-s)g_2\big(s,z(s)\big)d\upsilon(s),\hspace{2.9cm} t\in(e_q,t_{q+1}], ~q=1,2,\ldots,r.
	\end{array}\right.
\end{eqnarray*}
Now, we split $\mathcal{F}$ as $\mathcal{F}_1+\mathcal{F}_2,$ where 
\begin{eqnarray*}
	(\mathcal{F}_1 z)(t)=	\left \{ \begin{array}{lll}
		\mathcal{S}_{\alpha, \gamma_\iota}(t)z_0+\displaystyle\int_{0}^{t}\mathcal{S}_{\alpha, \gamma_\iota}(s)z_1ds+\sum \limits_{\iota=1}^{m}\beta_{\iota}\displaystyle\int_{0}^{t}\frac{(t-s)^{\alpha-\gamma_{\iota}}}{\Gamma(1+\alpha-\gamma_{\iota})}\mathcal{S}_{\alpha, \gamma_\iota}(s)z_0ds, \hspace{.1cm} t\in[0,t_1], q=0,&\\ \varsigma_q\big(t,z(t_q^-)\big), \hspace{6.7cm} t\in (t_q,e_q], ~q=1,2,\ldots,r,
		&\\	\mathcal{S}_{\alpha, \gamma_\iota}(t-e_q)\varsigma_q\big(e_q,z(t_q)\big)+\displaystyle\int_{e_q}^{t}\mathcal{S}_{\alpha, \gamma_\iota}(s-e_q)\varphi_q\big(e_q,z(t_q)\big)ds\nonumber\\ +\sum \limits_{\iota=1}^{m}\beta_{\iota}\displaystyle\int_{e_q}^{t}\frac{(t-s)^{\alpha-\gamma_{\iota}}}{\Gamma(1+\alpha-\gamma_{\iota})}\mathcal{S}_{\alpha, \gamma_\iota}(s-e_q)\varsigma_q\big(e_q,z(t_q)\big)ds,\hspace{.3cm} t\in(e_q,t_{q+1}], ~q=1,2,\ldots,r,
	\end{array}\right.
\end{eqnarray*}
and \begin{eqnarray*}
	(\mathcal{F}_2 z)(t)=	\left \{ \begin{array}{lll}
	\displaystyle\int_{0}^{t} \mathcal{J}_{\alpha, \gamma_\iota}(t-s)Eu(s)ds +\displaystyle\int_{0}^{t}\mathcal{J}_{\alpha, \gamma_\iota}(t-s) g_1\big(s,z(s)\big)ds\nonumber \\+\displaystyle\int_{0}^{t}\mathcal{J}_{\alpha, \gamma_\iota}(t-s)g_2\big(s,z(s)\big)d\upsilon(s), \hspace{3.2cm} t\in[0,t_1], ~q=0,&\\ 0, \hspace{8.1cm} t\in (t_q,e_q], ~q=1,2,\ldots,r,
		&\\	\displaystyle\int_{e_q}^{t} \mathcal{J}_{\alpha, \gamma_\iota}(t-s)Eu(s)ds +\displaystyle\int_{e_q}^{t}\mathcal{J}_{\alpha, \gamma_\iota}(t-s) g_1\big(s,z(s)\big)ds\nonumber \\+\displaystyle\int_{e_q}^{t}\mathcal{J}_{\alpha, \gamma_\iota}(t-s)g_2\big(s,z(s)\big)d\upsilon(s),\hspace{3.2cm} t\in(e_q,t_{q+1}], ~q=1,2,\ldots,r.
	\end{array}\right.
\end{eqnarray*}
For the sake of convenience, we divide the proof into several steps.\\
Step 1. There exists $h>0$ such that $\mathcal{F}(\Omega_h)\subset \Omega_h.$ Suppose this claim is false, we might choose $z^h\in \Omega_h$ for any $h>0,$  and $t\in [0,\ell]$ such that $\mathbb{E}\|\mathcal{F}(z^h)(t)\|^p>h.$ By Lemma~\ref{lm1} and H\"{o}lder's inequality, we have for $t\in[0,t_1],$
\begin{eqnarray*}
	h<\mathbb{E}\|\mathcal{F}(z^h)(t)\|^p&\leq& 6^{p-1}\mathbb{E}\|\mathcal{S}_{\alpha, \gamma_\iota}(t)z_0\|^p+6^{p-1}\mathbb{E}\bigg\|\displaystyle\int_{0}^{t}\mathcal{S}_{\alpha, \gamma_\iota}(s)z_1ds\bigg\|^p\\&& +6^{p-1}\mathbb{E}\bigg\|\sum \limits_{\iota=1}^{m}\beta_{\iota}\int_{0}^{t}\frac{(t-s)^{\alpha-\gamma_{\iota}}}{\Gamma(1+\alpha-\gamma_{\iota})}\mathcal{S}_{\alpha, \gamma_\iota}(s)z_0ds\bigg\|^p\\&& +6^{p-1}\mathbb{E}\bigg\|\int_{0}^{t} \mathcal{J}_{\alpha, \gamma_\iota}(t-s)Eu(s)ds\bigg\|^p \\&&+6^{p-1}\mathbb{E}\bigg\|\int_{0}^{t}\mathcal{J}_{\alpha, \gamma_\iota}(t-s) g_1\big(s,z^h(s)\big)ds\bigg\|^p \\&&+6^{p-1}\mathbb{E}\bigg\|\int_{0}^{t}\mathcal{J}_{\alpha, \gamma_\iota}(t-s)g_2\big(s,z^h(s)\big)d\upsilon(s)\bigg\|^p\\&\leq& 6^{p-1}S^p\mathbb{E}\|z_0\|^p+6^{p-1}S^p{t_1}\mathbb{E}\|z_1\|^p\\&&+6^{p-1}\sum \limits_{\iota=1}^{m}\beta_{\iota}\bigg(\frac{S}{\Gamma(1+\alpha-\gamma_{\iota})}\bigg)^p\frac{1}{(1+p \alpha-p \gamma_{\iota})}{t^{(1+\alpha-\gamma_{\iota})p}_1}\mathbb{E}\|z_0\|^p\\&&+6^{p-1}\mathbb{E}\bigg\|\int_{0}^{t} \mathcal{J}_{\alpha, \gamma_\iota}(t-s)Eu(s)ds\bigg\|^p \\&&+6^{p-1}\mathbb{E}\bigg\|\int_{0}^{t}\mathcal{J}_{\alpha, \gamma_\iota}(t-s) g_1\big(s,z^h(s)\big)ds\bigg\|^p \\&&+6^{p-1}C_p\mathbb{E}\bigg(\int_{0}^{t}\big\|\mathcal{J}_{\alpha, \gamma_\iota}(t-s)\big\|^2 \big\|g_2\big(s,z^h(s)\big)\big\|^2_{L_2^0}ds\bigg)^{\frac{p}{2}} \\&\leq& 6^{p-1}S^p\mathbb{E}\|z_0\|^p+6^{p-1}S^p{t_1}\mathbb{E}\|z_1\|^p\\&&+6^{p-1}\sum \limits_{\iota=1}^{m}\beta_{\iota}\bigg(\frac{S}{\Gamma(1+\alpha-\gamma_{\iota})}\bigg)^p\frac{1}{(1+p \alpha-p \gamma_{\iota})}{t^{(1+\alpha-\gamma_{\iota})p}_1}\mathbb{E}\|z_0\|^p\\&&+6^{p-1}\bigg(\frac{S}{\Gamma(1+\alpha)}\bigg)^p\bigg(\frac{p-1}{\alpha p+p-1}\bigg)^p{t^{\alpha p+p-1}_1}\|E\|^p\|u\|^p_{L_{\Upsilon}^p}\\&&+6^{p-1}\bigg(\frac{S}{\Gamma(1+\alpha)}\bigg)^p\bigg(\frac{p-1}{\alpha p+p-1}\bigg)^p{t^{\alpha p+p-1}_1}h\int_{0}^{t}\mu_1(s)ds\\&&+ 6^{p-1}C_p\bigg(\frac{S}{\Gamma(1+\alpha)}\bigg)^p\bigg(\frac{p-2}{2\alpha p+p-2}\bigg)^{\frac{p-2}{p}}{t^{\frac{2\alpha p+p-2}{2}}_1}h\int_{0}^{t}\mu_2(s)ds.
\end{eqnarray*}
For $t\in(t_q,e_q],q=1,2,\ldots,r,$ we obtain
\begin{eqnarray*}
	h<\mathbb{E}\|\mathcal{F}(z^h)(t)\|^p=\mathbb{E}\|\varsigma_q (t,z^h(t_q^-))\|^p\leq \tilde{a}_q(1+h).
\end{eqnarray*}
Similarly, for $t\in(e_q,t_{q+1}],q=1,2,\ldots,r,$ we obtain
\begin{eqnarray*}
		h<\mathbb{E}\|\mathcal{F}(z^h)(t)\|^p&\leq& 6^{p-1}\mathbb{E}\|\mathcal{S}_{\alpha, \gamma_\iota}(t-e_q)\varsigma_q\big(e_q,z^{h}(t_q)\big)\|^p\\&&+6^{p-1}\mathbb{E}\bigg\|\int_{e_q}^{t}\mathcal{S}_{\alpha, \gamma_\iota}(s-e_q)\varphi_q\big(e_q,z^h(t_q)\big)ds\bigg\|^p\\&& +6^{p-1}\mathbb{E}\bigg\|\sum \limits_{\iota=1}^{m}\beta_{\iota}\int_{e_q}^{t}\frac{(t-s)^{\alpha-\gamma_{\iota}}}{\Gamma(1+\alpha-\gamma_{\iota})}\mathcal{S}_{\alpha, \gamma_\iota}(s-e_q)\varsigma_q\big(e_q,z^h(t_q)\big)ds\bigg\|^p\\&&+6^{p-1}\mathbb{E}\bigg\|\int_{e_q}^{t} \mathcal{J}_{\alpha, \gamma_\iota}(t-s)Eu(s)ds\bigg\|^p \\&&+6^{p-1}\mathbb{E}\bigg\|\int_{e_q}^{t}\mathcal{J}_{\alpha, \gamma_\iota}(t-s) g_1\big(s,z^h(s)\big)ds\bigg\|^p \\&&+6^{p-1}\mathbb{E}\bigg\|\int_{e_q}^{t}\mathcal{J}_{\alpha, \gamma_\iota}(t-s)g_2\big(s,z^h(s)\big)d\upsilon(s)\bigg\|^p\\&\leq&6^{p-1}S^p\mathbb{E}\big\|\varsigma_q\big(e_q,z^{h}(t_q)\big)\big\|^p+6^{p-1}S^p\int_{e_q}^{t}\mathbb{E}\big\|\varphi_q\big(e_q,z^h(t_q)\big)\big\|^pds\\&& +6^{p-1}\mathbb{E}\bigg\|\sum \limits_{\iota=1}^{m}\beta_{\iota}\int_{e_q}^{t}\frac{(t-s)^{\alpha-\gamma_{\iota}}}{\Gamma(1+\alpha-\gamma_{\iota})}\mathcal{S}_{\alpha, \gamma_\iota}(s-e_q)\varsigma_q\big(e_q,z^h(t_q)\big)ds\bigg\|^p\\&&+6^{p-1}\mathbb{E}\bigg\|\int_{e_q}^{t} \mathcal{J}_{\alpha, \gamma_\iota}(t-s)Eu(s)ds\bigg\|^p \\&&+6^{p-1}\mathbb{E}\bigg\|\int_{e_q}^{t}\mathcal{J}_{\alpha, \gamma_\iota}(t-s) g_1\big(s,z^h(s)\big)ds\bigg\|^p \\&&+6^{p-1}C_p\mathbb{E}\bigg(\int_{e_q}^{t}\big\|\mathcal{J}_{\alpha, \gamma_\iota}(t-s)\big\|^2 \big\|g_2\big(s,z^h(s)\big)\big\|^2_{L_2^0}ds\bigg)^{\frac{p}{2}}\\&\leq& 6^{p-1}S^p\tilde{a}_q(1+h)+6^{p-1}S^pt_{q+1}\tilde{b}_q(1+h)\\&&+6^{p-1}\sum \limits_{\iota=1}^{m}\beta_{\iota}\bigg(\frac{S}{\Gamma(1+\alpha-\gamma_{\iota})}\bigg)^p\frac{1}{(1+p \alpha-p \gamma_{\iota})}{t^{(1+\alpha-\gamma_{\iota})p}_{q+1}}\tilde{a}_q(1+h)\\&&+6^{p-1}\bigg(\frac{S}{\Gamma(1+\alpha)}\bigg)^p\bigg(\frac{p-1}{\alpha p+p-1}\bigg)^p{t^{\alpha p+p-1}_{q+1}}\|E\|^p\|u\|^p_{L_{\Upsilon}^p}\\&&+6^{p-1}\bigg(\frac{S}{\Gamma(1+\alpha)}\bigg)^p\bigg(\frac{p-1}{\alpha p+p-1}\bigg)^p{t^{\alpha p+p-1}_{q+1}}h\int_{0}^{t}\mu_1(s)ds\\&&+ 6^{p-1}C_p\bigg(\frac{S}{\Gamma(1+\alpha)}\bigg)^p\bigg(\frac{p-2}{2\alpha p+p-2}\bigg)^{\frac{p-2}{p}}{t^{\frac{2\alpha p+p-2}{2}}_{q+1}}h\int_{0}^{t}\mu_2(s)ds.
\end{eqnarray*}
For any $t\in[0,\ell],$ we obtain
\begin{eqnarray*}
	h<\mathbb{E}\|\mathcal{F}(z^h)(t)\|^p&\leq& \mathcal{K}^*+\tilde{a}_qh+6^{p-1}S^p\tilde{a}_qh+6^{p-1}S^p\ell\tilde{b}_qh\\&&\quad+6^{p-1}\sum \limits_{\iota=1}^{m}\beta_{\iota}\bigg(\frac{S}{\Gamma(1+\alpha-\gamma_{\iota})}\bigg)^p\frac{1}{(1+p \alpha-p \gamma_{\iota})}{\ell^{(1+\alpha-\gamma_{\iota})p}}\tilde{a}_qh\\&&\quad+6^{p-1}\bigg(\frac{S}{\Gamma(1+\alpha)}\bigg)^p\bigg(\frac{p-1}{\alpha p+p-1}\bigg)^p{\ell^{\alpha p+p-1}}h\int_{0}^{t}\mu_1(s)ds\\&&\quad+ 6^{p-1}C_p\bigg(\frac{S}{\Gamma(1+\alpha)}\bigg)^p\bigg(\frac{p-2}{2\alpha p+p-2}\bigg)^{\frac{p-2}{p}}{\ell^{\frac{2\alpha p+p-2}{2}}}h\int_{0}^{t}\mu_2(s)ds,
\end{eqnarray*} 
where \begin{eqnarray*}\mathcal{K}^*&=&\max\limits_{1\leq q\leq r}\bigg\{6^{p-1}\sum \limits_{\iota=1}^{m}\beta_{\iota}\bigg(\frac{S}{\Gamma(1+\alpha-\gamma_{\iota})}\bigg)^p\frac{1}{(1+p \alpha-p \gamma_{\iota})}{\ell^{(1+\alpha-\gamma_{\iota})p}}\big(\mathbb{E}\|z_0\|^p+\tilde{a}_q\big)\\&&\hspace{1.5cm}+6^{p-1}\bigg(\frac{S}{\Gamma(1+\alpha)}\bigg)^p\bigg(\frac{p-1}{\alpha p+p-1}\bigg)^p{\ell^{\alpha p+p-1}}\|E\|^p\|u\|^p_{L_{\Upsilon}^p}\\&&\hspace{1.5cm}+6^{p-1}S^p\big(\mathbb{E}\|z_0\|^p+\tilde{a}_q+\ell\tilde{b}_q\big)+6^{p-1}S^p{\ell}\mathbb{E}\|z_1\|^p+\tilde{a}_q\bigg\}.
\end{eqnarray*} 
Here, $\mathcal{K}^*$ is independent of $h,$ both sides are dividing by $h$ and taking $h\rightarrow \infty,$ we obtain
\begin{eqnarray*}
	1&<&\tilde{a}_q+6^{p-1}S^p\tilde{a}_q+6^{p-1}S^p\ell\tilde{b}_q\\&&\quad+6^{p-1}\sum \limits_{\iota=1}^{m}\beta_{\iota}\bigg(\frac{S}{\Gamma(1+\alpha-\gamma_{\iota})}\bigg)^p\frac{1}{(1+p \alpha-p \gamma_{\iota})}{\ell^{(1+\alpha-\gamma_{\iota})p}}\tilde{a}_q\\&&\quad+6^{p-1}\bigg(\frac{S}{\Gamma(1+\alpha)}\bigg)^p\bigg(\frac{p-1}{\alpha p+p-1}\bigg)^p{\ell^{\alpha p+p-1}}\int_{0}^{t}\mu_1(s)ds\\&&\quad+ 6^{p-1}C_p\bigg(\frac{S}{\Gamma(1+\alpha)}\bigg)^p\bigg(\frac{p-2}{2\alpha p+p-2}\bigg)^{\frac{p-2}{p}}{\ell^{\frac{2\alpha p+p-2}{2}}}\int_{0}^{t}\mu_2(s)ds,
\end{eqnarray*} 
which contradicts to (C4). Hence, $ \mathcal{F}(\Omega_h)\subset \Omega_h,$ for some $h>0.$\\
Step 2. $\mathcal{F}_1$ is a contraction map on $\Omega_h.$ For any $y,z\in \Omega_h,$ if $t\in[0,t_1],$ then we have  
\begin{eqnarray}\label{2.7}
	\mathbb{E}\|(\mathcal{F}_1 y)(t)-(\mathcal{F}_1 z)(t)\|^p=0.
\end{eqnarray}
If $t\in(t_q,e_q],~q=1,2,\ldots,r,$ then we obtain
\begin{eqnarray}\label{2.8}
	\mathbb{E}\|(\mathcal{F}_1 y)(t)-(\mathcal{F}_1 z)(t)\|^p&=& 	\mathbb{E}\big\|\varsigma_q\big(t,y(t_q^-)\big)-\varsigma_q\big(t,z(t_q^-)\big)\big\|^p\nonumber\\&\leq& a_q\|y-z\|^p_{\mathcal{PC}}.
\end{eqnarray}
Similarly, if $t\in(e_q,t_{q+1}],~q=1,2,\ldots,r,$ then we obtain\\
$\mathbb{E}\|(\mathcal{F}_1 y)(t)-(\mathcal{F}_1 z)(t)\|^p$
\begin{eqnarray}\label{2.9}
&\leq& 3^{p-1}\mathbb{E}\bigg\|	\mathcal{S}_{\alpha, \gamma_\iota}(t-e_q)\big[\varsigma_q\big(e_q,y(t_q)\big)-\varsigma_q\big(e_q,z(t_q)\big)\big]\bigg\|^p\nonumber\\&&\quad+3^{p-1}\mathbb{E}\bigg\|\int_{e_q}^{t}\mathcal{S}_{\alpha, \gamma_\iota}(s-e_q)\big[\varphi_q\big(e_q,y(t_q)\big)-\varphi_q\big(e_q,z(t_q)\big)\big]ds\bigg\|^p\nonumber\\&&\quad +3^{p-1}\mathbb{E}\bigg\|\sum \limits_{\iota=1}^{m}\beta_{\iota}\int_{e_q}^{t}\frac{(t-s)^{\alpha-\gamma_{\iota}}}{\Gamma(1+\alpha-\gamma_{\iota})}\mathcal{S}_{\alpha, \gamma_\iota}(s-e_q)\big[\varsigma_q\big(e_q,y(t_q)\big)-\varsigma_q\big(e_q,z(t_q)\big)\big]ds\bigg\|^p\nonumber\\&\leq&3^{p-1}\bigg[S^pa_q+S^p{t_{q+1}}b_q+\sum \limits_{\iota=1}^{m}\beta_{\iota}\bigg(\frac{S}{\Gamma(1+\alpha-\gamma_{\iota})}\bigg)^p\frac{t^{(1+\alpha-\gamma_{\iota})p}_{q+1}}{(1+p \alpha-p \gamma_{\iota})}a_q\bigg]\|y-z\|^p_{\mathcal{PC}}.
\end{eqnarray}
From (\ref{2.7}) to (\ref{2.9}), we obtain
\begin{eqnarray*}
	\mathbb{E}\|(\mathcal{F}_1 y)(t)-(\mathcal{F}_1 z)(t)\|^p\leq \pounds_{\mathcal{F}_1}\|y-z\|^p_{\mathcal{PC}},
\end{eqnarray*}
where $\pounds_{\mathcal{F}_1}=\max\limits_{1\leq q\leq r}\bigg\{a_q,~ 3^{p-1}\bigg(S^pa_q+S^p{\ell}b_q+\sum \limits_{\iota=1}^{m}\beta_{\iota}\Big(\frac{S}{\Gamma(1+\alpha-\gamma_{\iota})}\Big)^p\frac{\ell^{(1+\alpha-\gamma_{\iota})p}}{(1+p \alpha-p \gamma_{\iota})}a_q\bigg)\bigg\}.$ By (C4), we see that $\pounds_{\mathcal{F}_1}<1.$ Hence, $\mathcal{F}_1$ is a contraction map.\\
Step 3. We prove that $\mathcal{F}_2$ is continuous on $\Omega_h$. Let $\{z^n\}_{n=1}^{\infty}$ be a sequence such that $z^n\rightarrow \hat{z}$ in $\Omega_h$ as $n \rightarrow \infty.$ By (C2), we have
\begin{eqnarray*} g_1(s,z^n(s))\rightarrow g_1(s,\hat z(s)),\\
	g_2(s,z^n(s))\rightarrow g_2(s,\hat z(s)),
~\mbox{as}~ n \rightarrow \infty,\end{eqnarray*}
for any $s\in[0,t]$ and \begin{eqnarray*}\mathbb{E}\|g_1(s,z^n)- g_1(s,\hat z)\|^p\leq 2h\mu_1(s),\\
\mathbb{E}\|g_2(s,z^n)- g_2(s,\hat z)\|^p_{L_2^0}\leq 2h\mu_2(s).\end{eqnarray*}
 For any $t\in(e_q,t_{q+1}],~q=0,1,2,\ldots,r,$ we obtain\\
 $	\mathbb{E}\|(\mathcal{F}_2 z^n)(t)-(\mathcal{F}_2\hat z)(t)\|^p	$
 \begin{eqnarray*}
&\leq& 2^{p-1}\mathbb{E}\bigg\|\int_{e_q}^{t}\mathcal{J}_{\alpha, \gamma_\iota}(t-s) \big[g_1\big(s,z^n(s)\big)-g_1\big(s,\hat z(s)\big)\big]ds\bigg\|^p\\&&\quad+2^{p-1}\mathbb{E}\bigg\|\int_{e_q}^{t}\mathcal{J}_{\alpha, \gamma_\iota}(t-s)\big[g_2\big(s,z^n(s)\big)-g_2\big(s,\hat z(s)\big)\big]d\upsilon(s)\bigg\|^p\\&\leq&2^{p-1}\mathbb{E}\bigg\|\int_{e_q}^{t}\mathcal{J}_{\alpha, \gamma_\iota}(t-s) \big[g_1\big(s,z^n(s)\big)-g_1\big(s,\hat z(s)\big)\big]ds\bigg\|^p\\&&\quad+2^{p-1}C_p\mathbb{E}\bigg(\int_{e_q}^{t}\big\|\mathcal{J}_{\alpha, \gamma_\iota}(t-s)\|^2\big\|g_2\big(s,z^n(s)\big)-g_2\big(s,\hat z(s)\big)\big\|^2_{L_2^0}ds\bigg)^{\frac{p}{2}}\\&\leq& 2^{p-1}\bigg(\frac{S}{\Gamma(1+\alpha)}\bigg)^p\bigg(\frac{p-1}{\alpha p+p-1}\bigg)^p{t^{\alpha p+p-1}_{q+1}}\int_{e_q}^{t}\mathbb{E}\big\|g_1\big(s,z^n(s)\big)-g_1\big(s,\hat z(s)\big)\big\|^pds\\&&\quad+2^{p-1}C_p\bigg(\frac{S}{\Gamma(1+\alpha)}\bigg)^p\bigg(\frac{p-2}{2\alpha p+p-2}\bigg)^{\frac{p-2}{p}}{t^{\frac{2\alpha p+p-2}{2}}_{q+1}}\int_{e_q}^{t}\mathbb{E}\big\|g_2\big(s,z^n(s)\big)-g_2\big(s,\hat z(s)\big)\big\|^pds.
\end{eqnarray*}	
By the Lebesgue dominated convergence theorem, we obtain
$$\|\mathcal{F}_2 z^n-\mathcal{F}_2\hat z\|^p_{\mathcal{PC}} \rightarrow 0 ~\mbox{as}~ n \rightarrow \infty.$$
Thus, $\mathcal{F}_2$ is continuous on $\Omega_h.$\\
Step 4. We show that $\{\mathcal{F}_2z:z\in \Omega_h\}$ is equicontinuous.\\
\noindent Let $z\in \Omega_h$ and $\tau_1,\tau_2\in (e_q,t_{q+1}].$ Then, if $e_q<\tau_1<\tau_2\leq t_{q+1},~q=0,1,2,\ldots,r,$ we obtain\\

$\mathbb{E}\|(\mathcal{F}_2z)(\tau_1)-(\mathcal{F}_2z)(\tau_2)\|^p$
\begin{eqnarray*}
	&\leq&6^{p-1}\mathbb{E}\bigg\|	\int_{\tau_1}^{\tau_2} \mathcal{J}_{\alpha, \gamma_\iota}(\tau_2-s)Eu(s)ds\bigg\|^p +6^{p-1}\mathbb{E}\bigg\|\int_{\tau_1}^{\tau_2}\mathcal{J}_{\alpha, \gamma_\iota}(\tau_2-s) g_1\big(s,z(s)\big)ds\bigg\|^p \\&&\quad+6^{p-1}\mathbb{E}\bigg\|\int_{\tau_1}^{\tau_2}\mathcal{J}_{\alpha, \gamma_\iota}(\tau_2-s)g_2\big(s,z(s)\big)d\upsilon(s)\bigg\|^p\\&& \quad+6^{p-1}\mathbb{E}\bigg\|\int_{e_q}^{\tau_1} \big[\mathcal{J}_{\alpha, \gamma_\iota}(\tau_2-s)-\mathcal{J}_{\alpha, \gamma_\iota}(\tau_1-s)\big]Eu(s)ds\bigg\|^p\\&&\quad +6^{p-1}\mathbb{E}\bigg\|\int_{e_q}^{\tau_1} \big[\mathcal{J}_{\alpha, \gamma_\iota}(\tau_2-s)-\mathcal{J}_{\alpha, \gamma_\iota}(\tau_1-s)\big]g_1\big(s,z(s)\big)ds\bigg\|^p \\&&\quad+6^{p-1}\mathbb{E}\bigg\|\int_{e_q}^{\tau_1} \big[\mathcal{J}_{\alpha, \gamma_\iota}(\tau_2-s)-\mathcal{J}_{\alpha, \gamma_\iota}(\tau_1-s)\big]g_2\big(s,z(s)\big)d\upsilon(s)\bigg\|^p\\
	&\leq&6^{p-1}\big[\mathcal{I}_1+\mathcal{I}_2+\mathcal{I}_3+\mathcal{I}_4+\mathcal{I}_5+\mathcal{I}_6\big]. 
\end{eqnarray*}
For terms $\mathcal{I}_i,~i=1,2$ and $3,$ we have
\begin{eqnarray*}
	\mathcal{I}_1&\leq&\bigg(\frac{S}{\Gamma(1+\alpha)}\bigg)^p\bigg(\frac{p-1}{\alpha p+p-1}\bigg)^p{t^{\alpha p+p-1}_{q+1}}\|E\|^p\mathbb{E}\int_{\tau_1}^{\tau_2}\|u(s)\|^pds \rightarrow 0~ \mbox{as}~\tau_2 \rightarrow \tau_1,\\
	\mathcal{I}_2&\leq& \bigg(\frac{S}{\Gamma(1+\alpha)}\bigg)^p\bigg(\frac{p-1}{\alpha p+p-1}\bigg)^p{t^{\alpha p+p-1}_{q+1}}h\int_{\tau_1}^{\tau_2}\mu_1(s)ds \rightarrow 0~ \mbox{as}~\tau_2 \rightarrow \tau_1,\\	\mathcal{I}_3&\leq& C_p\bigg(\frac{S}{\Gamma(1+\alpha)}\bigg)^p\bigg(\frac{p-2}{2\alpha p+p-2}\bigg)^{\frac{p-2}{p}}{t^{\frac{2\alpha p+p-2}{2}}_{q+1}}h\int_{\tau_1}^{\tau_2}\mu_2(s)ds \rightarrow 0~ \mbox{as}~\tau_2 \rightarrow \tau_1.
\end{eqnarray*}
For terms $\mathcal{I}_i,~i=4,5$ and $6,$ we have
\begin{eqnarray*}
		\mathcal{I}_4&\leq&\|E\|^p\mathbb{E}\int_{e_q}^{\tau_1} \big\|\mathcal{J}_{\alpha, \gamma_\iota}(\tau_2-s)-\mathcal{J}_{\alpha, \gamma_\iota}(\tau_1-s)\big\|^p\|u(s)\|^pds,\\\mathcal{I}_5&\leq&h\int_{e_q}^{\tau_1} \big\|\mathcal{J}_{\alpha, \gamma_\iota}(\tau_2-s)-\mathcal{J}_{\alpha, \gamma_\iota}(\tau_1-s)\big\|^p\mu_1(s)ds,\\\mathcal{I}_6&\leq&C_ph\int_{e_q}^{\tau_1} \big\|\mathcal{J}_{\alpha, \gamma_\iota}(\tau_2-s)-\mathcal{J}_{\alpha, \gamma_\iota}(\tau_1-s)\big\|^p\mu_2(s)ds.
\end{eqnarray*}
We see that the right-hand side of $\mathcal{I}_i\rightarrow 0$ as $\tau_2 \rightarrow \tau_1,~i=4,5,$ and $6,$ since the compactness of $\mathcal{J}_{\alpha, \gamma_\iota}(t)$ implies the continuity in the uniform operator topology. Hence, $\{\mathcal{F}_2z:z\in \Omega_h\}$ is equicontinuous. Moreover, $\{\mathcal{F}_2z:z\in \Omega_h\}$ is bounded.\\
Step 5. We show that $\mathcal{V}(t)=\{(\mathcal{F}_2z)(t):z\in \Omega_h\}$ is relatively compact in $\mathcal{Z}.$ Obviously, $\mathcal{V}(0)=\{0\}$ is compact. Let $\zeta$ be a real number and $t\in(e_q,t_{q+1}]$ be fixed with $0<\zeta<t.$ For $z\in \Omega_h,$ we define for all $t\in(e_q,t_{q+1}],~q=0,1,2,\ldots,r,$
\begin{eqnarray*}
	(\mathcal{F}^{\zeta}_2 z)(t)=	\left \{ \begin{array}{lll}
	 \displaystyle	\int_{0}^{t-\zeta} \mathcal{J}_{\alpha, \gamma_\iota}(t-s)Eu(s)ds +\displaystyle	\int_{0}^{t-\zeta} \mathcal{J}_{\alpha, \gamma_\iota}(t-s) g_1\big(s,z(s)\big)ds\nonumber \\+\displaystyle	\int_{0}^{t-\zeta} \mathcal{J}_{\alpha, \gamma_\iota}(t-s)g_2\big(s,z(s)\big)d\upsilon(s),  \hspace{2.4cm}t\in[0,t_1], q=0,&\\ 0, \hspace{7.7cm} t\in (t_q,e_q], q=1,2,\ldots,r,
		&\\	\displaystyle	\int_{e_q}^{t-\zeta}  \mathcal{J}_{\alpha, \gamma_\iota}(t-s)Eu(s)ds +\displaystyle	\int_{e_q}^{t-\zeta} \mathcal{J}_{\alpha, \gamma_\iota}(t-s) g_1\big(s,z(s)\big)ds\nonumber \\+\displaystyle	\int_{e_q}^{t-\zeta} \mathcal{J}_{\alpha, \gamma_\iota}(t-s)g_2\big(s,z(s)\big)d\upsilon(s), \hspace{2.5cm} t\in(e_q,t_{q+1}], q=1,2,\ldots,r.
	\end{array}\right.
\end{eqnarray*}
Since $\mathcal{J}_{\alpha, \gamma_\iota}(t)$ is compact, the set $\mathcal{V}^{\zeta}(t)=\big\{\big(\mathcal{F}^{\zeta}_2z\big)(t):z\in \Omega_h\big\}$ is relatively compact in $\mathcal{Z}$ for every $\zeta.$ Now, for each $0<\zeta<t$ and $t\in (e_q,t_{q+1}],~q=0,1,2,\ldots,r,$ we obtain\\
$\mathbb{E}\big\|(\mathcal{F}_2z)(t)-\big(\mathcal{F}^{\zeta}_2z\big)(t)\big\|^p$
\begin{eqnarray*}
	&\leq& 3^{p-1}\mathbb{E}\bigg\|\displaystyle	\int_{e_q}^{t}  \mathcal{J}_{\alpha, \gamma_\iota}(t-s)Eu(s)ds-\displaystyle	\int_{e_q}^{t-\zeta}  \mathcal{J}_{\alpha, \gamma_\iota}(t-s)Eu(s)ds\bigg\|^p\\&&\quad +3^{p-1}\mathbb{E}\bigg\|\displaystyle	\int_{e_q}^{t} \mathcal{J}_{\alpha, \gamma_\iota}(t-s) g_1\big(s,z(s)\big)ds-\displaystyle	\int_{e_q}^{t-\zeta} \mathcal{J}_{\alpha, \gamma_\iota}(t-s) g_1\big(s,z(s)\big)ds\bigg\|^p \\&&\quad+3^{p-1}\mathbb{E}\bigg\|\displaystyle	\int_{e_q}^{t} \mathcal{J}_{\alpha, \gamma_\iota}(t-s)g_2\big(s,z(s)\big)d\upsilon(s)-\displaystyle	\int_{e_q}^{t-\zeta} \mathcal{J}_{\alpha, \gamma_\iota}(t-s)g_2\big(s,z(s)\big)d\upsilon(s)\bigg\|^p\\&\leq&3^{p-1}\bigg(\frac{S}{\Gamma(1+\alpha)}\bigg)^p\bigg(\frac{p-1}{\alpha p+p-1}\bigg)^p{t^{\alpha p+p-1}_{q+1}}\bigg[\|E\|^p\mathbb{E}\displaystyle	\int_{t-\zeta}^{t} \|u(s)\|^pds+h\displaystyle	\int_{t-\zeta}^{t}\mu_1(s)ds\bigg]\\&&\quad+3^{p-1}C_p\bigg(\frac{S}{\Gamma(1+\alpha)}\bigg)^p\bigg(\frac{p-2}{2\alpha p+p-2}\bigg)^{\frac{p-2}{p}}{t^{\frac{2\alpha p+p-2}{2}}_{q+1}}h\displaystyle	\int_{t-\zeta}^{t}\mu_2(s)ds\rightarrow 0~\mbox{as}~ \zeta \rightarrow 0.
\end{eqnarray*} 
As a result, $\mathcal{V}(t)$ is relatively compact in $\mathcal{Z}.$  $\mathcal{F}_2 $ is completely continuous according to the Arzela-Ascoli theorem and steps 3 to 5. Hence, by Krasnoselskii's fixed point theorem \cite{rrr}, $\mathcal{F}$ has at least one fixed point on $\Omega_h,$ which is a mild solution of the system (\ref{1.1}). 
\end{proof}
	
\section{\textbf{Existence of Stochastic Optimal Controls}}
Let $\mathfrak{U}_z(u)$ represent the collection of all solutions of (\ref{1.1}) and $z^u$ denotes the mild solution of (\ref{1.1}) with respect to $u\in U_{ad}.$ The Lagrange problem is used to find an optimal state-control pair $(z^*,u^*)\in \mathcal{PC(Z)} \times {U}_{ad}$ satisfying \cite {r1}
$$	\mathcal{I}(z^*,u^*)\leq \mathcal{I}(z^u,u),\forall~ u\in {U}_{ad},$$
where 
\begin{eqnarray} \label{05.1}
	\mathcal{I}(z^u,u)=\mathbb{E} \bigg\{\int_{0}^{\ell}\tilde{\mathcal{G}}\big(t,z^u(t),u(t)\big)dt\bigg\}.
\end{eqnarray}
The following hypotheses are used to discuss the Lagrange problem:
\begin{itemize}
	\item [(C5)] The Borel measurable function $\tilde {\mathcal{G}}:[0, \ell] \times \mathcal{Z} \times \mathcal{U} \rightarrow \mathbb{R}\cup \{\infty\}$ fulfills
	\subitem(a) For almost all $t\in[0,\ell],~\tilde {\mathcal{G}}(t,z,\cdot)$ is convex on $\mathcal{U}$ for each $z \in \mathcal{Z}.$
	\subitem(b) For almost all $t\in[0,\ell],~\tilde {\mathcal{G}}(t,\cdot,\cdot)$ is sequentially lower semicontinuous on $\mathcal{Z} \times\mathcal{U}.$ 
	\subitem(c) There exist constants $h_1\geq 0,~~h_2>0$ and $\Phi$ is a non-negative function in $\mathcal{L}^1([0,\ell],\mathbb{R})$ such that
	$$\tilde{\mathcal{G}}\big(t,z(t),u(t)\big)\geq \Phi(t)+h_1\|z\|^p+h_2\|u\|^p_{\mathcal{U}}.$$
\end{itemize}
Denote $$\Delta\geq \max\limits_{1\leq q\leq r}\bigg[(x_1+x_{10})\exp(x_{30}),~\frac{\tilde{a}_q}{1-\tilde{a}_q},~(x_{2q}+x_{1q})(1+\tilde{\pounds})^q\exp(x_{3q})\bigg],$$
where $x_1=6^{p-1}S^p\mathbb{E}\|z_0\|^p+6^{p-1}S^pt_1\mathbb{E}\|z_1\|^p+6^{p-1}\sum \limits_{\iota=1}^{m}\beta_{\iota}\Big(\frac{S}{\Gamma(1+\alpha-\gamma_{\iota})}\Big)^p\frac{1}{(1+p \alpha-p \gamma_{\iota})}{t^{(1+\alpha-\gamma_{\iota})p}_1}\mathbb{E}\|z_0\|^p,$\\ $x_{1q}=6^{p-1}\Big(\frac{S}{\Gamma(1+\alpha)}\Big)^p\Big(\frac{p-1}{\alpha p+p-1}\Big)^p{t^{\alpha p+p-1}_{q+1}}\|E\|^p\|u\|^p_{L_{\Upsilon}^p},$\\ $x_{3q}=6^{p-1}\Big(\frac{S}{\Gamma(1+\alpha)}\Big)^p\bigg\{\Big(\frac{p-1}{\alpha p+p-1}\Big)^p{t^{\alpha p+p-1}_{q+1}}+\Big(\frac{p-2}{2\alpha p+p-2}\Big)^{\frac{p-2}{p}}{t^{\frac{2\alpha p+p-2}{2}}_{q+1}}\bigg\}\|\mu_1(s)+\mu_2(s)\|_{L^1},$  \\ $x_{2q}=6^{p-1}S^p\tilde{a}_q+6^{p-1}{t_{q+1}}S^p\tilde{b}_q+6^{p-1}\sum \limits_{\iota=1}^{m}\beta_{\iota}\Big(\frac{S}{\Gamma(1+\alpha-\gamma_{\iota})}\Big)^p\frac{1}{(1+p \alpha-p \gamma_{\iota})}{t^{(1+\alpha-\gamma_{\iota})p}_{q+1}}\tilde{a}_q,\quad$ $\tilde{a_q}<1,$ and\\  $\tilde{\pounds}=\sup\limits_{1\leq k\leq r}6^{p-1}S^p\bigg[\tilde{a}_k+{t_{q+1}}\tilde{b}_k+\sum \limits_{\iota=1}^{m}\beta_{\iota}\Big(\frac{1}{\Gamma(1+\alpha-\gamma_{\iota})}\Big)^p\frac{1}{(1+p \alpha-p \gamma_{\iota})}{t^{(1+\alpha-\gamma_{\iota})p}_{q+1}}\tilde{a}_k\bigg].$
\begin{lemma}\label{3.1}
	If the hypotheses (C1)-(C4) are satisfied, then for any mild solution $z$ of (\ref{1.1}) there exists a constant $\Delta>0$ such that $\|z\|^p_{\mathcal{PC}}\leq \Delta$ for given $u\in U_{ad}.$ 
\end{lemma}
\begin{proof}
	Let $z$ be a mild solution of (\ref{1.1}) with respect to $u \in U_{ad}$ on $[0,\ell].$ The proof is divided into several cases.\\
	Case 1. For any $t\in[0,t_1],$ we have
	\begin{eqnarray*}
		\mathbb{E}\|z(t)\|^p&\leq&6^{p-1}\mathbb{E}\Big\|	\mathcal{S}_{\alpha, \gamma_\iota}(t)z_0\Big\|^p+6^{p-1}\mathbb{E}\bigg\|\int_{0}^{t}\mathcal{S}_{\alpha, \gamma_\iota}(s)z_1ds\bigg\|^p\\&& \quad+6^{p-1}\mathbb{E}\bigg\|\sum\limits_{\iota=1}^{m}\beta_{\iota}\int_{0}^{t}\frac{(t-s)^{\alpha-\gamma_{\iota}}}{\Gamma(1+\alpha-\gamma_{\iota})}\mathcal{S}_{\alpha, \gamma_\iota}(s)z_0ds\bigg\|^p\\&& \quad+6^{p-1}\mathbb{E}\bigg\|\int_{0}^{t} \mathcal{J}_{\alpha, \gamma_\iota}(t-s)Eu(s)ds\bigg\|^p +6^{p-1}\mathbb{E}\bigg\|\int_{0}^{t}\mathcal{J}_{\alpha, \gamma_\iota}(t-s) g_1\big(s,z(s)\big)ds\bigg\|^p \\&& \quad+6^{p-1}\mathbb{E}\bigg\|\int_{0}^{t}\mathcal{J}_{\alpha, \gamma_\iota}(t-s)g_2\big(s,z(s)\big)d\upsilon(s)\bigg\|^p
		\\&\leq& 6^{p-1}S^p\mathbb{E}\|z_0\|^p+6^{p-1}S^p{t_1}\mathbb{E}\|z_1\|^p\\&&\quad+6^{p-1}\sum \limits_{\iota=1}^{m}\beta_{\iota}\bigg(\frac{S}{\Gamma(1+\alpha-\gamma_{\iota})}\bigg)^p\frac{1}{(1+p \alpha-p \gamma_{\iota})}{t^{(1+\alpha-\gamma_{\iota})p}_1}\mathbb{E}\|z_0\|^p\\&&\quad+6^{p-1}\bigg(\frac{S}{\Gamma(1+\alpha)}\bigg)^p\bigg(\frac{p-1}{\alpha p+p-1}\bigg)^p{t^{\alpha p+p-1}_1}\|E\|^p\|u\|^p_{L_{\Upsilon}^p}\\&&\quad+6^{p-1}\bigg(\frac{S}{\Gamma(1+\alpha)}\bigg)^p\bigg(\frac{p-1}{\alpha p+p-1}\bigg)^p{t^{\alpha p+p-1}_1}\int_{0}^{t}\mu_1(s)\mathbb{E}\|z(s)\|^pds\\&&\quad+ 6^{p-1}C_p\bigg(\frac{S}{\Gamma(1+\alpha)}\bigg)^p\bigg(\frac{p-2}{2\alpha p+p-2}\bigg)^{\frac{p-2}{p}}{t^{\frac{2\alpha p+p-2}{2}}_1}\int_{0}^{t}\mu_2(s)\mathbb{E}\|z(s)\|^pds.
	\end{eqnarray*}
By Gronwall's inequality, we obtain
\begin{eqnarray*}
	\mathbb{E}\|z(t)\|^p&\leq&\bigg\{6^{p-1}S^p\mathbb{E}\|z_0\|^p+6^{p-1}S^p{t_1}\mathbb{E}\|z_1\|^p\\&&\quad+6^{p-1}\sum \limits_{\iota=1}^{m}\beta_{\iota}\bigg(\frac{S}{\Gamma(1+\alpha-\gamma_{\iota})}\bigg)^p\frac{1}{(1+p \alpha-p \gamma_{\iota})}{t^{(1+\alpha-\gamma_{\iota})p}_1}\mathbb{E}\|z_0\|^p\\&&\quad+6^{p-1}\bigg(\frac{S}{\Gamma(1+\alpha)}\bigg)^p\bigg(\frac{p-1}{\alpha p+p-1}\bigg)^p{t^{\alpha p+p-1}_1}\|E\|^p\|u\|^p_{L_{\Upsilon}^p}\bigg\}\\&&\quad\times\exp\bigg[6^{p-1}\bigg(\frac{S}{\Gamma(1+\alpha)}\bigg)^p\bigg\{\bigg(\frac{p-1}{\alpha p+p-1}\bigg)^p{t^{\alpha p+p-1}_1}\\&&\quad+\bigg(\frac{p-2}{2\alpha p+p-2}\bigg)^{\frac{p-2}{p}}{t^{\frac{2\alpha p+p-2}{2}}_1}\bigg\}\|\mu_1(s)+\mu_2(s)\|_{L^1}\bigg]\\&\leq&\Delta.
\end{eqnarray*}
Case 2. For any $t\in(t_q,e_q],~q=1,2,\ldots,r,$ we obtain
\begin{eqnarray*}
		\mathbb{E}\|z(t)\|^p&=&\mathbb{E}\|\varsigma_q (t,z(t_q^-))\|^p\leq \tilde{a}_q\big(1+\mathbb{E}\|z(t_q^-)\|^p\big),
\end{eqnarray*}
which yields that
\begin{eqnarray*}
	\|z\|^p_{\mathcal{PC}}\leq \frac{\tilde{a}_q}{1-\tilde{a}_q} \leq \Delta.
\end{eqnarray*}
Case 3. For any $t\in(e_q,t_{q+1}],~q=1,2,\ldots,r,$ we obtain
\begin{eqnarray*}
	\mathbb{E}\|z(t)\|^p&\leq&6^{p-1}\mathbb{E}\Big\|\mathcal{S}_{\alpha, \gamma_\iota}(t-e_q)\varsigma_q\big(e_q,z(t_q)\big)\Big\|^p+6^{p-1}\mathbb{E}\bigg\|\int_{e_q}^{t}\mathcal{S}_{\alpha, \gamma_\iota}(s-e_q)\varphi_q\big(e_q,z(t_q)\big)ds\bigg\|^p\\&&\quad +6^{p-1}\mathbb{E}\bigg\|\sum \limits_{\iota=1}^{m}\beta_{\iota}\int_{e_q}^{t}\frac{(t-s)^{\alpha-\gamma_{\iota}}}{\Gamma(1+\alpha-\gamma_{\iota})}\mathcal{S}_{\alpha, \gamma_\iota}(s-e_q)\varsigma_q\big(e_q,z(t_q)\big)ds\bigg\|^p\\&& \quad+6^{p-1}\mathbb{E}\bigg\|\int_{e_q}^{t} \mathcal{J}_{\alpha, \gamma_\iota}(t-s)Eu(s)ds\bigg\|^p +6^{p-1}\mathbb{E}\bigg\|\int_{e_q}^{t}\mathcal{J}_{\alpha, \gamma_\iota}(t-s) g_1\big(s,z(s)\big)ds\bigg\|^p \\&&\quad+6^{p-1}\mathbb{E}\bigg\|\int_{e_q}^{t}\mathcal{J}_{\alpha, \gamma_\iota}(t-s)g_2\big(s,z(s)\big)d\upsilon(s)\bigg\|^p\\&\leq& 6^{p-1}S^p\tilde{a}_q\big(1+\mathbb{E}\|z(t_q)\|^p\big)+6^{p-1}{t_{q+1}}S^p\tilde{b}_q\big(1+\mathbb{E}\|z(t_q)\|^p\big)\\&&\quad +6^{p-1}\sum \limits_{\iota=1}^{m}\beta_{\iota}\bigg(\frac{S}{\Gamma(1+\alpha-\gamma_{\iota})}\bigg)^p\frac{1}{(1+p \alpha-p \gamma_{\iota})}{t^{(1+\alpha-\gamma_{\iota})p}_{q+1}}\tilde{a}_q\big(1+\mathbb{E}\|z(t_q)\|^p\big)\\&&\quad+6^{p-1}\bigg(\frac{S}{\Gamma(1+\alpha)}\bigg)^p\bigg(\frac{p-1}{\alpha p+p-1}\bigg)^p{t^{\alpha p+p-1}_{q+1}}\|E\|^p\|u\|^p_{L_{\Upsilon}^p}\\&&\quad+6^{p-1}\bigg(\frac{S}{\Gamma(1+\alpha)}\bigg)^p\bigg(\frac{p-1}{\alpha p+p-1}\bigg)^p{t^{\alpha p+p-1}_{q+1}}\int_{0}^{t}\mu_1(s)\mathbb{E}\|z(s)\|^pds\\&&\quad+ 6^{p-1}C_p\bigg(\frac{S}{\Gamma(1+\alpha)}\bigg)^p\bigg(\frac{p-2}{2\alpha p+p-2}\bigg)^{\frac{p-2}{p}}{t^{\frac{2\alpha p+p-2}{2}}_{q+1}}\int_{0}^{t}\mu_2(s)\mathbb{E}\|z(s)\|^pds\\&\leq& 6^{p-1}S^p\tilde{a}_q+6^{p-1}{t_{q+1}}S^p\tilde{b}_q+6^{p-1}\sum \limits_{\iota=1}^{m}\beta_{\iota}\bigg(\frac{S}{\Gamma(1+\alpha-\gamma_{\iota})}\bigg)^p\frac{1}{(1+p \alpha-p \gamma_{\iota})}{t^{(1+\alpha-\gamma_{\iota})p}_{q+1}}\tilde{a}_q\\&&\quad+6^{p-1}\bigg(\frac{S}{\Gamma(1+\alpha)}\bigg)^p\bigg(\frac{p-1}{\alpha p+p-1}\bigg)^p{t^{\alpha p+p-1}_{q+1}}\|E\|^p\|u\|^p_{L_{\Upsilon}^p}\\&&\quad+6^{p-1}\bigg(\frac{S}{\Gamma(1+\alpha)}\bigg)^p\bigg\{\bigg(\frac{p-1}{\alpha p+p-1}\bigg)^p{t^{\alpha p+p-1}_{q+1}}+\bigg(\frac{p-2}{2\alpha p+p-2}\bigg)^{\frac{p-2}{p}}{t^{\frac{2\alpha p+p-2}{2}}_{q+1}}\bigg\}\\&&\quad\times \int_{0}^{t}\big(\mu_1(s)+\mu_2(s)\big)\mathbb{E}\|z(s)\|^pds+\sum\limits_{k=1}^{q}6^{p-1}S^p\bigg[\tilde{a}_k+{t_{q+1}}\tilde{b}_k\\&&\quad+\sum \limits_{\iota=1}^{m}\beta_{\iota}\bigg(\frac{1}{\Gamma(1+\alpha-\gamma_{\iota})}\bigg)^p\frac{1}{(1+p \alpha-p \gamma_{\iota})}{t^{(1+\alpha-\gamma_{\iota})p}_{q+1}}\tilde{a}_k\bigg]\mathbb{E}\|z(t_k)\|^p.
\end{eqnarray*}
Using impulsive Gronwall's inequality (Theorem 16.1, \cite{DP}), we get
\begin{eqnarray*}
	\mathbb{E}\|z(t)\|^p&\leq&\bigg\{6^{p-1}S^p\tilde{a}_q+6^{p-1}{t_{q+1}}S^p\tilde{b}_q+6^{p-1}\sum \limits_{\iota=1}^{m}\beta_{\iota}\bigg(\frac{S}{\Gamma(1+\alpha-\gamma_{\iota})}\bigg)^p\frac{1}{(1+p \alpha-p \gamma_{\iota})}{t^{(1+\alpha-\gamma_{\iota})p}_{q+1}}\tilde{a}_q\\&&\quad+6^{p-1}\bigg(\frac{S}{\Gamma(1+\alpha)}\bigg)^p\bigg(\frac{p-1}{\alpha p+p-1}\bigg)^p{t^{\alpha p+p-1}_{q+1}}\|E\|^p\|u\|^p_{L_{\Upsilon}^p}\bigg\}\\&&\quad \times (1+\tilde{\pounds})^q\exp\bigg[6^{p-1}\bigg(\frac{S}{\Gamma(1+\alpha)}\bigg)^p\bigg\{\bigg(\frac{p-1}{\alpha p+p-1}\bigg)^p{t^{\alpha p+p-1}_{q+1}}\\&&\quad+\bigg(\frac{p-2}{2\alpha p+p-2}\bigg)^{\frac{p-2}{p}}{t^{\frac{2\alpha p+p-2}{2}}_{q+1}}\bigg\}\|\mu_1(s)+\mu_2(s)\|_{L^1}\bigg]\\&\leq&\Delta,
\end{eqnarray*}
where $\tilde{\pounds}=\sup\limits_{1\leq k\leq r}6^{p-1}S^p\bigg[\tilde{a}_k+{t_{q+1}}\tilde{b}_k+\sum \limits_{\iota=1}^{m}\beta_{\iota}\bigg(\frac{1}{\Gamma(1+\alpha-\gamma_{\iota})}\bigg)^p\frac{1}{(1+p \alpha-p \gamma_{\iota})}{t^{(1+\alpha-\gamma_{\iota})p}_{q+1}}\tilde{a}_k\bigg].$\\
From the above, we have
$$\|z\|^p_{\mathcal{PC}}\leq \Delta.$$
\end{proof}
\begin{theorem} \label{th4} Assume that the assumptions (C1)-(C5) are satisfied. Then the Lagrange problem (\ref{05.1}) admits at least one optimal state-control pair.
\end{theorem}
\begin{proof}
We define $\mathcal{I}(u)=\inf_{z^u\in \mathfrak{U}_z(u)} \mathcal{I}(z^u,u)$ for $u \in U_{ad}.$ If the collection $\mathfrak{U}_z(u)$ has a finite number of elements, there exists some $\hat z^u \in \mathfrak{U}_z(u)$ such that $\mathcal{I}(\hat z^u,u)=\inf_{z^u\in \mathfrak{U}_z(u)}\mathcal{I}(z^u,u)=\mathcal{I}(u).$ If the collection $\mathfrak{U}_z(u)$ admits infinite many elements and $\inf_{z^u\in \mathfrak{U}_z(u)}\mathcal{I}(z^u,u)=+\infty,$ there is nothing to prove. Next, we choose $\mathcal{I}(u)=\inf_{z^u\in \mathfrak{U}_z(u)} \mathcal{I}(z^u,u)<+\infty$ and by using the hypotheses (C5), we get $\mathcal{I}(u)>-\infty.$ The proof is now divided into three steps. \\
Step 1. There exists a minimizing sequence $\{z^u_m\}\subseteq  \mathfrak{U}_z(u)$ such that
$$ \mathcal{I}(z^u_m,u) \rightarrow \mathcal{I}(u)~\mbox{as}~ m \rightarrow \infty.$$
Let $\{z^u_m\}$ represent the sequence of mild solutions of (\ref{1.1}) with respect to $u$
 	\begin{eqnarray*}
 	z^u_m(t)=	\left \{ \begin{array}{lll}
 		\mathcal{S}_{\alpha, \gamma_\iota}(t)z_0+\displaystyle\int_{0}^{t}\mathcal{S}_{\alpha, \gamma_\iota}(s)z_1ds+\sum \limits_{\iota=1}^{m}\beta_{\iota}\displaystyle\int_{0}^{t}\frac{(t-s)^{\alpha-\gamma_{\iota}}}{\Gamma(1+\alpha-\gamma_{\iota})}\mathcal{S}_{\alpha, \gamma_\iota}(s)z_0ds\nonumber\\ +\displaystyle\int_{0}^{t} \mathcal{J}_{\alpha, \gamma_\iota}(t-s)Eu(s)ds +\displaystyle\int_{0}^{t}\mathcal{J}_{\alpha, \gamma_\iota}(t-s) g_1\big(s,z^u_m(s)\big)ds\nonumber \\+\displaystyle\int_{0}^{t}\mathcal{J}_{\alpha, \gamma_\iota}(t-s)g_2\big(s,z^u_m(s)\big)d\upsilon(s), \hspace{3.1cm} t\in[0,t_1], ~q=0,&\\ \varsigma_q\big(t,z^u_m(t_q^-)\big), \hspace{6.4cm} t\in (t_q,e_q], ~q=1,2,\ldots,r,
 		&\\	\mathcal{S}_{\alpha, \gamma_\iota}(t-e_q)\varsigma_q\big(e_q,z^u_m(t_q)\big)+\displaystyle\int_{e_q}^{t}\mathcal{S}_{\alpha, \gamma_\iota}(s-e_q)\varphi_q\big(e_q,z^u_m(t_q)\big)ds\nonumber\\ +\sum \limits_{\iota=1}^{m}\beta_{\iota}\displaystyle\int_{e_q}^{t}\frac{(t-s)^{\alpha-\gamma_{\iota}}}{\Gamma(1+\alpha-\gamma_{\iota})}\mathcal{S}_{\alpha, \gamma_\iota}(s-e_q)\varsigma_q\big(e_q,z^u_m(t_q)\big)ds\nonumber\\+\displaystyle\int_{e_q}^{t} \mathcal{J}_{\alpha, \gamma_\iota}(t-s)Eu(s)ds +\displaystyle\int_{e_q}^{t}\mathcal{J}_{\alpha, \gamma_\iota}(t-s) g_1\big(s,z^u_m(s)\big)ds\nonumber \\+\displaystyle\int_{e_q}^{t}\mathcal{J}_{\alpha, \gamma_\iota}(t-s)g_2\big(s,z^u_m(s)\big)d\upsilon(s),\hspace{2.9cm} t\in(e_q,t_{q+1}], ~q=1,2,\ldots,r.
 	\end{array}\right.
  \end{eqnarray*}
Step 2. Next, we prove that there exists some $\hat z^u\in \mathfrak{U}_z(u) $ such that $\mathcal{I}(\hat z^u,u)=\inf_{z^u\in \mathfrak{U}_z(u)}\mathcal{I}(z^u,u)=\mathcal{I}(u).$ For this, we show that for each $u\in U_{ad}$, the set $\{z^u_m\}_{m\in \mathbb{N}}$ is relatively compact in $\mathcal{PC(Z)}.$ $\{z^u_m\}$ is uniformly bounded according to Lemma ~\ref{3.1}. We now demonstrate that $\{z^u_m\}$ is equicontinuous on $[0,\ell].$  We consider the following three scenarios to achieve our goal:\\
Case 1. For $0<\zeta_1<\zeta_2\leq t_1,$ we obtain
\begin{eqnarray*}
	\mathbb{E}\|z(\zeta_2)-z(\zeta_1)\|^p&\leq& {10}^{p-1}\|\mathcal{S}_{\alpha, \gamma_\iota}(\zeta_2)-\mathcal{S}_{\alpha, \gamma_\iota}(\zeta_1)\|^p\mathbb{E}\|z_0\|^p+{10}^{p-1}\mathbb{E}\bigg\|\displaystyle\int_{\zeta_1}^{\zeta_2}\mathcal{S}_{\alpha, \gamma_\iota}(s)z_1ds\bigg\|^p\\&&\quad+{10}^{p-1}\mathbb{E}\bigg\|\sum \limits_{\iota=1}^{m}\beta_{\iota}\displaystyle\int_{\zeta_1}^{\zeta_2}\frac{(\zeta_2-s)^{\alpha-\gamma_{\iota}}}{\Gamma(1+\alpha-\gamma_{\iota})}\mathcal{S}_{\alpha, \gamma_\iota}(s)z_0ds\bigg\|^p\\&& \quad+{10}^{p-1}\mathbb{E}\bigg\|\sum \limits_{\iota=1}^{m}\beta_{\iota}\displaystyle\int_{0}^{\zeta_1}\frac{[(\zeta_2-s)^{\alpha-\gamma_{\iota}}-(\zeta_1-s)^{\alpha-\gamma_{\iota}}]}{\Gamma(1+\alpha-\gamma_{\iota})}\mathcal{S}_{\alpha, \gamma_\iota}(s)z_0ds\bigg\|^p\\&&\quad+{10}^{p-1}\mathbb{E}\bigg\|\displaystyle\int_{\zeta_1}^{\zeta_2} \mathcal{J}_{\alpha, \gamma_\iota}(\zeta_2-s)Eu(s)ds\bigg\|^p\\&&\quad+{10}^{p-1}\mathbb{E}\bigg\|\displaystyle\int_{0}^{\zeta_1} \big[\mathcal{J}_{\alpha, \gamma_\iota}(\zeta_2-s)-\mathcal{J}_{\alpha, \gamma_\iota}(\zeta_1-s)\big]Eu(s)ds\bigg\|^p \\&&\quad +{10}^{p-1}\mathbb{E}\bigg\|\displaystyle\int_{\zeta_1}^{\zeta_2}\mathcal{J}_{\alpha, \gamma_\iota}(\zeta_2-s) g_1\big(s,z(s)\big)ds\bigg\|^p \\&&\quad+{10}^{p-1}\mathbb{E}\bigg\|\displaystyle\int_{0}^{\zeta_1}\big[\mathcal{J}_{\alpha, \gamma_\iota}(\zeta_2-s)-\mathcal{J}_{\alpha, \gamma_\iota}(\zeta_1-s)\big] g_1\big(s,z(s)\big)ds\bigg\|^p \\&&\quad+{10}^{p-1}\mathbb{E}\bigg\|\displaystyle\int_{\zeta_1}^{\zeta_2}\mathcal{J}_{\alpha, \gamma_\iota}(\zeta_2-s)g_2\big(s,z(s)\big)d\upsilon(s)\bigg\|^p\\&&\quad+{10}^{p-1}\mathbb{E}\bigg\|\displaystyle\int_{0}^{\zeta_1}\big[\mathcal{J}_{\alpha, \gamma_\iota}(\zeta_2-s)-\mathcal{J}_{\alpha, \gamma_\iota}(\zeta_1-s)\big]g_2\big(s,z(s)\big)d\upsilon(s)\bigg\|^p\\&\leq& {10}^{p-1}\bigg[\|\mathcal{S}_{\alpha, \gamma_\iota}(\zeta_2)-\mathcal{S}_{\alpha, \gamma_\iota}(\zeta_1)\|^p\mathbb{E}\|z_0\|^p+(\zeta_2-\zeta_1)S^p\mathbb{E}\big\|z_1\big\|^p\\&&\quad+ \Lambda_1+\Lambda_2+\Lambda_3+\Lambda_4\bigg],
\end{eqnarray*}
where 
\begin{eqnarray*}
	\Lambda_1&=&\mathbb{E}\bigg\|\sum \limits_{\iota=1}^{m}\beta_{\iota}\displaystyle\int_{\zeta_1}^{\zeta_2}\frac{(\zeta_2-s)^{\alpha-\gamma_{\iota}}}{\Gamma(1+\alpha-\gamma_{\iota})}\mathcal{S}_{\alpha, \gamma_\iota}(s)z_0ds\bigg\|^p\\&& \quad+\mathbb{E}\bigg\|\sum \limits_{\iota=1}^{m}\beta_{\iota}\displaystyle\int_{0}^{\zeta_1}\frac{[(\zeta_2-s)^{\alpha-\gamma_{\iota}}-(\zeta_1-s)^{\alpha-\gamma_{\iota}}]}{\Gamma(1+\alpha-\gamma_{\iota})}\mathcal{S}_{\alpha, \gamma_\iota}(s)z_0ds\bigg\|^p\\&\leq&\sum \limits_{\iota=1}^{m}\beta_{\iota}\bigg(\frac{S}{\Gamma(1+\alpha-\gamma_{\iota})}\bigg)^p\frac{1}{(1+p \alpha-p \gamma_{\iota})}{t^{(1+\alpha-\gamma_{\iota})p}_1}\mathbb{E}\|z_0\|^p \\&& \quad+\sum \limits_{\iota=1}^{m}\beta_{\iota}\mathbb{E}\displaystyle\int_{0}^{\zeta_1}\frac{\big\|(\zeta_2-s)^{\alpha-\gamma_{\iota}}-(\zeta_1-s)^{\alpha-\gamma_{\iota}}\big\|^p}{\Gamma(1+\alpha-\gamma_{\iota})}\big\|\mathcal{S}_{\alpha, \gamma_\iota}(s)\big\|^p\|z_0\|^pds,\\
	\Lambda_2&=&\mathbb{E}\bigg\|\displaystyle\int_{\zeta_1}^{\zeta_2} \mathcal{J}_{\alpha, \gamma_\iota}(\zeta_2-s)Eu(s)ds\bigg\|^p\\&&\quad+\mathbb{E}\bigg\|\displaystyle\int_{0}^{\zeta_1} \big[\mathcal{J}_{\alpha, \gamma_\iota}(\zeta_2-s)-\mathcal{J}_{\alpha, \gamma_\iota}(\zeta_1-s)\big]Eu(s)ds\bigg\|^p \\&\leq& \bigg(\frac{S}{\Gamma(1+\alpha)}\bigg)^p\bigg(\frac{p-1}{\alpha p+p-1}\bigg)^p{t^{\alpha p+p-1}_1}\|E\|^p\mathbb{E}\displaystyle\int_{\zeta_1}^{\zeta_2} \|u(s)\|^pds\\&&\quad+\mathbb{E}\displaystyle\int_{0}^{\zeta_1} \big\|\mathcal{J}_{\alpha, \gamma_\iota}(\zeta_2-s)-\mathcal{J}_{\alpha, \gamma_\iota}(\zeta_1-s)\big\|^p\|E\|^p\|u(s)\|^pds,\\
	\Lambda_3&=&\mathbb{E}\bigg\|\displaystyle\int_{\zeta_1}^{\zeta_2}\mathcal{J}_{\alpha, \gamma_\iota}(\zeta_2-s) g_1\big(s,z(s)\big)ds\bigg\|^p \\&&\quad+\mathbb{E}\bigg\|\displaystyle\int_{0}^{\zeta_1}\big[\mathcal{J}_{\alpha, \gamma_\iota}(\zeta_2-s)-\mathcal{J}_{\alpha, \gamma_\iota}(\zeta_1-s)\big] g_1\big(s,z(s)\big)ds\bigg\|^p\\&\leq&\bigg(\frac{S}{\Gamma(1+\alpha)}\bigg)^p\bigg(\frac{p-1}{\alpha p+p-1}\bigg)^p{t^{\alpha p+p-1}_1}\Delta\displaystyle\int_{\zeta_1}^{\zeta_2}\mu_1(s)ds\\&&\quad+\Delta\displaystyle\int_{0}^{\zeta_1}\big\|\mathcal{J}_{\alpha, \gamma_\iota}(\zeta_2-s)-\mathcal{J}_{\alpha, \gamma_\iota}(\zeta_1-s)\big\|^p\mu_1(s)ds,\\
	\Lambda_4&=&\mathbb{E}\bigg\|\displaystyle\int_{\zeta_1}^{\zeta_2}\mathcal{J}_{\alpha, \gamma_\iota}(\zeta_2-s)g_2\big(s,z(s)\big)d\upsilon(s)\bigg\|^p\\&&\quad+\mathbb{E}\bigg\|\displaystyle\int_{0}^{\zeta_1}\big[\mathcal{J}_{\alpha, \gamma_\iota}(\zeta_2-s)-\mathcal{J}_{\alpha, \gamma_\iota}(\zeta_1-s)\big]g_2\big(s,z(s)\big)d\upsilon(s)\bigg\|^p\\&\leq& C_p\bigg(\frac{S}{\Gamma(1+\alpha)}\bigg)^p\bigg(\frac{p-2}{2\alpha p+p-2}\bigg)^{\frac{p-2}{p}}{t^{\frac{2\alpha p+p-2}{2}}_1}\Delta\displaystyle\int_{\zeta_1}^{\zeta_2}\mu_2(s)ds\\&&\quad+C_p\Delta\displaystyle\int_{0}^{\zeta_1}\big\|\mathcal{J}_{\alpha, \gamma_\iota}(\zeta_2-s)-\mathcal{J}_{\alpha, \gamma_\iota}(\zeta_1-s)\big\|^p\mu_2(s)ds.
\end{eqnarray*}
Case 2. For $t_q<\zeta_1<\zeta_2\leq e_q,~q =1,2,\ldots,r,$ we obtain
\begin{eqnarray*}
\mathbb{E}\|z(\zeta_2)-z(\zeta_1)\|^p=	\mathbb{E}\|\varsigma_q\big(\zeta_2,z(t_q^-)\big)-\varsigma_q\big(\zeta_1,z(t_q^-)\big)\|^p.
\end{eqnarray*}
Case 3. For $e_q<\zeta_1<\zeta_2\leq t_{q+1},~q =1,2,\ldots,r,$ we obtain\\
$\mathbb{E}\|z(\zeta_2)-z(\zeta_1)\|^p$
\begin{eqnarray*}
	&\leq& 	{10}^{p-1}\|\mathcal{S}_{\alpha, \gamma_\iota}(\zeta_2-e_q)-\mathcal{S}_{\alpha, \gamma_\iota}(\zeta_1-e_q)\|^p\mathbb{E}\|\varsigma_q\big(e_q,z(t_q)\big)\|^p\\&&\quad+{10}^{p-1}\mathbb{E}\bigg\|\displaystyle\int_{\zeta_1}^{\zeta_2}\mathcal{S}_{\alpha, \gamma_\iota}(s-e_q)\varphi_q\big(e_q,z(t_q)\big)ds\bigg\|^p\\&&\quad +{10}^{p-1}\mathbb{E}\bigg\|\sum \limits_{\iota=1}^{m}\beta_{\iota}\displaystyle\int_{\zeta_1}^{\zeta_2}\frac{(\zeta_2-s)^{\alpha-\gamma_{\iota}}}{\Gamma(1+\alpha-\gamma_{\iota})}\mathcal{S}_{\alpha, \gamma_\iota}(s-e_q)\varsigma_q\big(e_q,z(t_q)\big)ds\bigg\|^p\\&&\quad +{10}^{p-1}\mathbb{E}\bigg\|\sum \limits_{\iota=1}^{m}\beta_{\iota}\displaystyle\int_{e_q}^{\zeta_1}\frac{[(\zeta_2-s)^{\alpha-\gamma_{\iota}}-(\zeta_1-s)^{\alpha-\gamma_{\iota}}]}{\Gamma(1+\alpha-\gamma_{\iota})}\mathcal{S}_{\alpha, \gamma_\iota}(s-e_q)\varsigma_q\big(e_q,z(t_q)\big)ds\bigg\|^p\\&& \quad+{10}^{p-1}\mathbb{E}\bigg\|\displaystyle\int_{\zeta_1}^{\zeta_2} \mathcal{J}_{\alpha, \gamma_\iota}(\zeta_2-s)Eu(s)ds\bigg\|^p\\&& \quad+{10}^{p-1}\mathbb{E}\bigg\|\displaystyle\int_{e_q}^{\zeta_1} [\mathcal{J}_{\alpha, \gamma_\iota}(\zeta_2-s)-\mathcal{J}_{\alpha, \gamma_\iota}(\zeta_1-s) ]Eu(s)ds\bigg\|^p\\&&\quad +{10}^{p-1}\mathbb{E}\bigg\|\displaystyle\int_{\zeta_1}^{\zeta_2} \mathcal{J}_{\alpha, \gamma_\iota}(\zeta_2-s)g_1\big(s,z(s)\big)ds\bigg\|^p\\&&\quad+{10}^{p-1}\mathbb{E}\bigg\|\displaystyle\int_{e_q}^{\zeta_1}[\mathcal{J}_{\alpha, \gamma_\iota}(\zeta_2-s)-\mathcal{J}_{\alpha, \gamma_\iota}(\zeta_1-s) ] g_1\big(s,z(s)\big)ds\bigg\|^p \\&&\quad+{10}^{p-1}\mathbb{E}\bigg\|\displaystyle\int_{\zeta_1}^{\zeta_2}\mathcal{J}_{\alpha, \gamma_\iota}(\zeta_2-s)g_2\big(s,z(s)\big)d\upsilon(s)\bigg\|^p\\&&\quad+{10}^{p-1}\mathbb{E}\bigg\|\displaystyle\int_{e_q}^{\zeta_1}[\mathcal{J}_{\alpha, \gamma_\iota}(\zeta_2-s)-\mathcal{J}_{\alpha, \gamma_\iota}(\zeta_1-s) ] g_2\big(s,z(s)\big)d\upsilon(s)\bigg\|^p\\	&\leq& 	{10}^{p-1}\bigg[\|\mathcal{S}_{\alpha, \gamma_\iota}(\zeta_2-e_q)-\mathcal{S}_{\alpha, \gamma_\iota}(\zeta_1-e_q)\|^p\mathbb{E}\|\varsigma_q\big(e_q,z(t_q)\big)\|^p+(\zeta_2-\zeta_1)S^p\tilde{b_q}(1+\Delta)\\&&\quad+\Lambda'_1+\Lambda'_2+\Lambda'_3+\Lambda'_4\bigg],
	\end{eqnarray*}
where 
\begin{eqnarray*}
	\Lambda'_1&=&\mathbb{E}\bigg\|\sum \limits_{\iota=1}^{m}\beta_{\iota}\displaystyle\int_{\zeta_1}^{\zeta_2}\frac{(\zeta_2-s)^{\alpha-\gamma_{\iota}}}{\Gamma(1+\alpha-\gamma_{\iota})}\mathcal{S}_{\alpha, \gamma_\iota}(s-e_q)\varsigma_q\big(e_q,z(t_q)\big)ds\bigg\|^p\\&&\quad +\mathbb{E}\bigg\|\sum \limits_{\iota=1}^{m}\beta_{\iota}\displaystyle\int_{e_q}^{\zeta_1}\frac{[(\zeta_2-s)^{\alpha-\gamma_{\iota}}-(\zeta_1-s)^{\alpha-\gamma_{\iota}}]}{\Gamma(1+\alpha-\gamma_{\iota})}\mathcal{S}_{\alpha, \gamma_\iota}(s-e_q)\varsigma_q\big(e_q,z(t_q)\big)ds\bigg\|^p\\&\leq&\sum \limits_{\iota=1}^{m}\beta_{\iota}\bigg(\frac{S}{\Gamma(1+\alpha-\gamma_{\iota})}\bigg)^p\frac{1}{(1+p \alpha-p \gamma_{\iota})}{t^{(1+\alpha-\gamma_{\iota})p}_{q+1}}\tilde{a_q}(1+\Delta)\\&& \quad+S^p\tilde{b_q}(1+\Delta)\sum \limits_{\iota=1}^{m}\beta_{\iota}\mathbb{E}\displaystyle\int_{e_q}^{\zeta_1}\frac{\big\|(\zeta_2-s)^{\alpha-\gamma_{\iota}}-(\zeta_1-s)^{\alpha-\gamma_{\iota}}\big\|^p}{\Gamma(1+\alpha-\gamma_{\iota})}ds,\\
	\Lambda'_2&=&\mathbb{E}\bigg\|\displaystyle\int_{\zeta_1}^{\zeta_2} \mathcal{J}_{\alpha, \gamma_\iota}(\zeta_2-s)Eu(s)ds\bigg\|^p\\&&\quad+\mathbb{E}\bigg\|\displaystyle\int_{e_q}^{\zeta_1} \big[\mathcal{J}_{\alpha, \gamma_\iota}(\zeta_2-s)-\mathcal{J}_{\alpha, \gamma_\iota}(\zeta_1-s)\big]Eu(s)ds\bigg\|^p \\&\leq&\bigg(\frac{S}{\Gamma(1+\alpha)}\bigg)^p\bigg(\frac{p-1}{\alpha p+p-1}\bigg)^p{t^{\alpha p+p-1}_{q+1}}\|E\|^p\mathbb{E}\displaystyle\int_{\zeta_1}^{\zeta_2} \|u(s)\|^pds\\&&\quad+\|E\|^p\mathbb{E}\displaystyle\int_{e_q}^{\zeta_1} \big\|\mathcal{J}_{\alpha, \gamma_\iota}(\zeta_2-s)-\mathcal{J}_{\alpha, \gamma_\iota}(\zeta_1-s)\big\|^p\|u(s)\|^pds,\\
	\Lambda'_3&=&\mathbb{E}\bigg\|\displaystyle\int_{\zeta_1}^{\zeta_2}\mathcal{J}_{\alpha, \gamma_\iota}(\zeta_2-s) g_1\big(s,z(s)\big)ds\bigg\|^p \\&&\quad+\mathbb{E}\bigg\|\displaystyle\int_{e_q}^{\zeta_1}\big[\mathcal{J}_{\alpha, \gamma_\iota}(\zeta_2-s)-\mathcal{J}_{\alpha, \gamma_\iota}(\zeta_1-s)\big] g_1\big(s,z(s)\big)ds\bigg\|^p\\&\leq&\bigg(\frac{S}{\Gamma(1+\alpha)}\bigg)^p\bigg(\frac{p-1}{\alpha p+p-1}\bigg)^p{t^{\alpha p+p-1}_{q+1}}\Delta\displaystyle\int_{\zeta_1}^{\zeta_2}\mu_1(s)ds\\&&\quad+\Delta\displaystyle\int_{e_q}^{\zeta_1}\big\|\mathcal{J}_{\alpha, \gamma_\iota}(\zeta_2-s)-\mathcal{J}_{\alpha, \gamma_\iota}(\zeta_1-s)\big\|^p\mu_1(s)ds,\\
	\Lambda'_4&=&\mathbb{E}\bigg\|\displaystyle\int_{\zeta_1}^{\zeta_2}\mathcal{J}_{\alpha, \gamma_\iota}(\zeta_2-s)g_2\big(s,z(s)\big)d\upsilon(s)\bigg\|^p\\&&\quad+\mathbb{E}\bigg\|\displaystyle\int_{e_q}^{\zeta_1}\big[\mathcal{J}_{\alpha, \gamma_\iota}(\zeta_2-s)-\mathcal{J}_{\alpha, \gamma_\iota}(\zeta_1-s)\big]g_2\big(s,z(s)\big)d\upsilon(s)\bigg\|^p\\&\leq& C_p\bigg(\frac{S}{\Gamma(1+\alpha)}\bigg)^p\bigg(\frac{p-2}{2\alpha p+p-2}\bigg)^{\frac{p-2}{p}}{t^{\frac{2\alpha p+p-2}{2}}_{q+1}}\Delta\displaystyle\int_{\zeta_1}^{\zeta_2}\mu_2(s)ds\\&&\quad+C_p\Delta\displaystyle\int_{e_q}^{\zeta_1}\big\|\mathcal{J}_{\alpha, \gamma_\iota}(\zeta_2-s)-\mathcal{J}_{\alpha, \gamma_\iota}(\zeta_1-s)\big\|^p\mu_2(s)ds.
\end{eqnarray*}
Using the hypotheses (C1) and continuity of $\varsigma_q$, we can see that the right-hand side of $\Lambda'_i,~i=1,2,3,4$, tends to zero as $\zeta_2\rightarrow \zeta_1.$ Hence,  $\{z^u_m\}$ is equicontinuous on $[0,\ell].$ Similar to the proof of steps 4 and 5 in Theorem~\ref{th1}, we can easily see that $\{z^u_m\}$ is relatively compact on $\mathcal{PC(Z)}.$ As a result, for $u \in U_{ad}$ there exists $\hat z^u \in \mathcal{PC(Z)}$ such that $z^u_m \rightarrow \hat z^u$ in $\mathcal{PC(Z)}.$ Applying $m \rightarrow \infty,$ we obtain\\
\begin{eqnarray*}
\hat z^u(t)=	\left \{ \begin{array}{lll}
		\mathcal{S}_{\alpha, \gamma_\iota}(t)z_0+\displaystyle\int_{0}^{t}\mathcal{S}_{\alpha, \gamma_\iota}(s)z_1ds+\sum \limits_{\iota=1}^{m}\beta_{\iota}\displaystyle\int_{0}^{t}\frac{(t-s)^{\alpha-\gamma_{\iota}}}{\Gamma(1+\alpha-\gamma_{\iota})}\mathcal{S}_{\alpha, \gamma_\iota}(s)z_0ds\nonumber\\ +\displaystyle\int_{0}^{t} \mathcal{J}_{\alpha, \gamma_\iota}(t-s)Eu(s)ds +\displaystyle\int_{0}^{t}\mathcal{J}_{\alpha, \gamma_\iota}(t-s) g_1\big(s,\hat z^u(s)\big)ds\nonumber \\+\displaystyle\int_{0}^{t}\mathcal{J}_{\alpha, \gamma_\iota}(t-s)g_2\big(s,\hat z^u(s)\big)d\upsilon(s), \hspace{3.1cm} t\in[0,t_1], ~q=0,&\\ \varsigma_q\big(t,\hat z^u(t_q^-)\big), \hspace{6.5cm} t\in (t_q,e_q], ~q=1,2,\ldots,r,
		&\\	\mathcal{S}_{\alpha, \gamma_\iota}(t-e_q)\varsigma_q\big(e_q,\hat z^u(t_q)\big)+\displaystyle\int_{e_q}^{t}\mathcal{S}_{\alpha, \gamma_\iota}(s-e_q)\varphi_q\big(e_q,\hat z^u(t_q)\big)ds\nonumber\\ +\sum \limits_{\iota=1}^{m}\beta_{\iota}\displaystyle\int_{e_q}^{t}\frac{(t-s)^{\alpha-\gamma_{\iota}}}{\Gamma(1+\alpha-\gamma_{\iota})}\mathcal{S}_{\alpha, \gamma_\iota}(s-e_q)\varsigma_q\big(e_q, \hat z^u(t_q)\big)ds\nonumber\\+\displaystyle\int_{e_q}^{t} \mathcal{J}_{\alpha, \gamma_\iota}(t-s)Eu(s)ds +\displaystyle\int_{e_q}^{t}\mathcal{J}_{\alpha, \gamma_\iota}(t-s) g_1\big(s,\hat z^u(s)\big)ds\nonumber \\+\displaystyle\int_{e_q}^{t}\mathcal{J}_{\alpha, \gamma_\iota}(t-s)g_2\big(s,\hat z^u(s)\big)d\upsilon(s),\hspace{2.9cm} t\in(e_q,t_{q+1}], ~q=1,2,\ldots,r,
	\end{array}\right.
\end{eqnarray*}
which implies that $\hat z^u\in \mathfrak{U}_z(u).$ Now, we claim that $\mathcal{I}(\hat z^u,u)=\inf_{z^u\in \mathfrak{U}_z(u)}\mathcal{I}(z^u,u)=\mathcal{I}(u).$ By using the Balder's theorem (see Theorem 2.1, \cite{e1}) and hypotheses (C5), we get
	\begin{eqnarray*}
\mathcal{I}(u)&=&\lim \limits_{m \rightarrow \infty} \mathbb{E} \bigg\{\int_{0}^{\ell}\tilde{\mathcal{G}}\big(t,z^u_m(t),u(t)\big)dt\bigg\} \\
	& \geq & \mathbb{E} \bigg\{\int_{0}^{\ell}\tilde{\mathcal{G}}\big(t,\hat z^u(t),u(t)\big)dt\bigg\}\\
	& =& \mathcal{I}(\hat z^u,u) \geq \mathcal{I}(u),
\end{eqnarray*} 
which yields that $\mathcal{I}(u)$ attains its minimum at $\hat z^u \in \mathcal{PC(Z)}$ for each $u \in U_{ad}.$\\
Step 3. Next, we show that there exists $u^*\in U_{ad}$ such that $$\mathcal{I}(u^*)\leq \mathcal{I}(u)~\mbox{for all}~ u \in U_{ad}.$$
If $\inf_{u \in U_{ad}}\mathcal{I}(u)=+\infty,$ there is nothing to prove. Now, we suppose that $\inf_{u \in U_{ad}}\mathcal{I}(u)<+\infty.$ Using the hypotheses (C5) again, we obtain $\inf_{u \in U_{ad}}\mathcal{I}(u)>-\infty.$ By definition of infimum, there exists a minimizing sequence $\{u_m\}\subseteq U_{ad}$ such that
$$\mathcal{I}(u_m)\rightarrow \inf_{u\in U_{ad}} \mathcal{I}(u)~\mbox{as} ~ m \rightarrow +\infty.$$
Since $\{u_m\}\subseteq U_{ad},~\{u_m\}$ is bounded in $\mathcal L_{\Upsilon}^p([0,\ell],\mathcal{U}),$ there exists a subsequence, still denoted by $\{u_m\}$ and $u^*\in \mathcal L_{\Upsilon}^p([0,\ell],\mathcal{U})$ such that $u_m$ converges weakly to $u^*$ in $\mathcal L_{\Upsilon}^p([0,\ell],\mathcal{U})$ as $m \rightarrow +\infty.$ Since $U_{ad}$ is convex and closed, then by Marzur Lemma, $u^*\in U_{ad}.$ Let $\{\hat z^{u_m}\}$ represent the sequence of mild solutions of (\ref{1.1}) with respect to $\{u_m\}$
\begin{eqnarray*}
	\hat z^{u_m}(t)=	\left \{ \begin{array}{lll}
		\mathcal{S}_{\alpha, \gamma_\iota}(t)z_0+\displaystyle\int_{0}^{t}\mathcal{S}_{\alpha, \gamma_\iota}(s)z_1ds+\sum \limits_{\iota=1}^{m}\beta_{\iota}\displaystyle\int_{0}^{t}\frac{(t-s)^{\alpha-\gamma_{\iota}}}{\Gamma(1+\alpha-\gamma_{\iota})}\mathcal{S}_{\alpha, \gamma_\iota}(s)z_0ds\nonumber\\ +\displaystyle\int_{0}^{t} \mathcal{J}_{\alpha, \gamma_\iota}(t-s)Eu_m(s)ds +\displaystyle\int_{0}^{t}\mathcal{J}_{\alpha, \gamma_\iota}(t-s) g_1\big(s,\hat z^{u_m}(s)\big)ds\nonumber \\+\displaystyle\int_{0}^{t}\mathcal{J}_{\alpha, \gamma_\iota}(t-s)g_2\big(s,\hat z^{u_m}(s)\big)d\upsilon(s), \hspace{3.1cm} t\in[0,t_1], ~q=0,&\\ \varsigma_q\big(t,\hat z^{u_m}(t_q^-)\big), \hspace{6.5cm} t\in (t_q,e_q], ~q=1,2,\ldots,r,
		&\\	\mathcal{S}_{\alpha, \gamma_\iota}(t-e_q)\varsigma_q\big(e_q,\hat z^{u_m}(t_q)\big)+\displaystyle\int_{e_q}^{t}\mathcal{S}_{\alpha, \gamma_\iota}(s-e_q)\varphi_q\big(e_q,\hat z^{u_m}(t_q)\big)ds\nonumber\\ +\sum \limits_{\iota=1}^{m}\beta_{\iota}\displaystyle\int_{e_q}^{t}\frac{(t-s)^{\alpha-\gamma_{\iota}}}{\Gamma(1+\alpha-\gamma_{\iota})}\mathcal{S}_{\alpha, \gamma_\iota}(s-e_q)\varsigma_q\big(e_q, \hat z^{u_m}(t_q)\big)ds\nonumber\\ +\displaystyle\int_{e_q}^{t} \mathcal{J}_{\alpha, \gamma_\iota}(t-s)Eu_m(s)ds +\displaystyle\int_{e_q}^{t}\mathcal{J}_{\alpha, \gamma_\iota}(t-s) g_1\big(s,\hat z^{u_m}(s)\big)ds\nonumber \\+\displaystyle\int_{e_q}^{t}\mathcal{J}_{\alpha, \gamma_\iota}(t-s)g_2\big(s,\hat z^{u_m}(s)\big)d\upsilon(s),\hspace{2.8cm} t\in(e_q,t_{q+1}], ~q=1,2,\ldots,r.
	\end{array}\right.
\end{eqnarray*}
We can easily demonstrate that $\{\hat z^{u_m}\}$ is relatively compact on $\mathcal{PC(Z)}$ similar to the proof in the step 2.  As a result, there exists $\hat z^{u^*} \in \mathcal{PC(Z)}$ such that $\hat z^{u_m} \rightarrow \hat z^{u^*}$ in $\mathcal{PC(Z)}$ for $u^* \in U_{ad}.$ Taking $m \rightarrow \infty,$ we obtain
\begin{eqnarray*}
&&	\varsigma_q\big(t,\hat z^{u_m}(t_q^-)\big)\rightarrow \varsigma_q\big(t,\hat z^{u^*}(t_q^-)\big),~q=1,2,\ldots,r,\\&&\varphi_q\big(t,\hat z^{u_m}(t_q^-)\big)\rightarrow \varphi_q\big(t,\hat z^{u^*}(t_q^-)\big),~q=1,2,\ldots,r,\\&&\displaystyle\int_{e_q}^{t} \mathcal{J}_{\alpha, \gamma_\iota}(t-s)Eu_m(s)ds\rightarrow \displaystyle\int_{e_q}^{t} \mathcal{J}_{\alpha, \gamma_\iota}(t-s)Eu^*(s)ds,~ q=0,1,2,\ldots,r,\\&&
\displaystyle\int_{e_q}^{t}\mathcal{J}_{\alpha, \gamma_\iota}(t-s) g_1\big(s,\hat z^{u_m}(s)\big)ds\rightarrow\displaystyle\int_{e_q}^{t}\mathcal{J}_{\alpha, \gamma_\iota}(t-s) g_1\big(s,\hat z^{u^*}(s)\big)ds,~ q=0,1,2,\ldots,r,\\&&\displaystyle\int_{e_q}^{t}\mathcal{J}_{\alpha, \gamma_\iota}(t-s) g_2\big(s,\hat z^{u_m}(s)\big)d\upsilon(s)\rightarrow\displaystyle\int_{e_q}^{t}\mathcal{J}_{\alpha, \gamma_\iota}(t-s) g_2\big(s,\hat z^{u^*}(s)\big)d\upsilon(s),~ q=0,1,2,\ldots,r.
	\end{eqnarray*}
Hence, $\hat z^{u^*}$ represents the solution of (\ref{1.1}) with respect to $u^*.$
By using the Balder's theorem (see Theorem 2.1, \cite{e1}) and hypotheses (C5), we get
\begin{eqnarray*}
	\inf_{u\in U{ad}}\mathcal{I}(u)&=&\lim \limits_{m \rightarrow \infty} \mathbb{E} \bigg\{\int_{0}^{\ell}\tilde{\mathcal{G}}\big(t,\hat z^{u_m}(t),u_m(t)\big)dt\bigg\} \\
	& \geq & \mathbb{E} \bigg\{\int_{0}^{\ell}\tilde{\mathcal{G}}\big(t,\hat z^{u^*}(t),u^*(t)\big)dt\bigg\}\\
	& =& \mathcal{I}(\hat z^{u^*},u^*) \geq \inf_{u\in U{ad}}\mathcal{I}(u).
\end{eqnarray*} 
Thus, $\mathcal{I}(\hat z^{u^*},u^*)=\mathcal{I}(u^*)=\inf_{z^{u^*}\in \mathfrak{U}_z(u^*)}\mathcal{I}(z^{u^*},u^*).$ Moreover, $\mathcal{I}(u^*)=\inf_{u\in U_{ad}}\mathcal{I}(u),$ i.e., $\mathcal{I}$ attains its minimum at $u^*\in U_{ad}.$
	\end{proof}

\section{\textbf{Applications}}
\textbf{Example 6.1.} Consider 
\begin{eqnarray}  \label{5.1}
\left \{ \begin{array}{lll} ^cD_{t}^{1+\alpha}z(t,x)+\sum \limits_{i=1}^{m}\beta_i ^cD_{t}^{\gamma_i}z(t,x)= z_{xx}(t,x)+u(t,x)+\dfrac{e^{-3t}z(t,x)}{2+|z(t,x)|}\\\hspace{6cm}+\dfrac{e^{\pi}z(t,x)}{e^{4t}}\dfrac{d\upsilon(t)}{dt},\\
	\hspace{5.5cm}\quad t\in (0,0.20]\cup (0.90,1],\quad u\in U_{ad},\quad x \in [0,\pi],
&\\z(t,x)=\frac{1}{4}\sin (t)z(0.20^-,x), \quad t \in (0.20,0.90],\quad x \in [0,\pi],
&\\z'(t,x)=\frac{1}{3}\sin(t)z(0.20^-,x), \quad t \in (0.20,0.90],\quad x \in [0,\pi],
&\\z(t,0)=z(t,\pi)=0, &\\z(0,x)=z_0(x), \quad \dfrac{\partial z(t,x)}{\partial t}\Big\rvert_{t=0}=z_1(x) , \quad t \in [0,\pi],
\end{array}\right.
\end{eqnarray}
with cost functional as
\begin{eqnarray*}
	\mathcal{I}(z,u)=\int_{0}^{1}\int_{0}^{\pi}\|z^u(t,x)\|^2dxdt+\int_{0}^{1}\int_{0}^{\pi}\|u(t,x)\|^2dxdt,
\end{eqnarray*}
where,  $0<\alpha \leq \gamma_m\leq \cdots \leq \gamma_1\leq 1$ and $\beta_i\geq0$ be given for all $i=1,2,\ldots, m$ and $\upsilon(t)$ is a Wiener process. $0=e_0=t_0<t_1<e_1<t_2=\ell,$ with $t_1=0.20,e_1=0.90,$ and $t_2=1.$  Let $\mathcal{Z}=\mathcal{U}=\mathcal{L}^2[0,\pi],~\mathcal{W}=\mathbb{R},$ and define $A:D(A)\subset \mathcal{Z}\rightarrow \mathcal{Z}$
 by
 $Az=z''$
with $$D(A)=\{z \in \mathcal{Z} : z,z'~ \mbox{are absolutely continuous},~ z'' \in \mathcal{Z}, \mbox{and}~z(0)=z(\pi)=0\}.$$ 
The operator $A$ has discrete spectrum with normalized eigenvectors $\xi_n(x)=\sqrt{\frac{2}{\pi}}\sin (nx) $ corresponding to the eigenvalues $\lambda_n=-n^2,$ where $n \in \mathbb{N}.$ Moreover, $\{\xi_n:n \in \mathbb{N}\}$ forms an orthogonal basis for $\mathcal{Z}.$ Thus, we have $$Az=\sum \limits_{n \in \mathbb{N}}{-n^2} \langle z,\xi_n \rangle \xi_n, ~ z\in D(A).$$ 
$A$ generates a strongly continuous cosine and sine family given by
$$C(t)z=\sum \limits_{n \in \mathbb{N}}\cos(nt)\langle z,\xi_n \rangle \xi_n,$$ and $$S(t)z=\sum \limits_{n \in \mathbb{N}}\frac{1}{n}\sin(nt)\langle z,\xi_n \rangle \xi_n,~z\in \mathcal{Z},$$ respectively. 
From Huan \cite{H}, $z\in \mathcal{Z}$ and $t\in \mathbb{R},$ $\|C(t)\|\leq 1$ and $\|S(t)\|\leq 1.$
In view of Theorem~\ref{1}, for $\alpha \in(0,1)$ and $\gamma_i>0,$ $A$ generates a bounded $(\alpha, \gamma_i)$-resolvent family $$ S_{\alpha, \gamma_i}(t)z=\int_0^\infty \frac{1}{t^{\frac{(1+\alpha)}{2}}}\psi_{\frac{(1+\alpha)}{2}}\Big(st^{-\frac{(1+\alpha)}{2}}\Big)C(s)zds,~~t\in[0,1],$$
where $$\psi_{\frac{(1+\alpha)}{2}}(l)=\sum\limits_{n=0}^{\infty}\frac{(-l)^n}{n!\Gamma(-(\alpha(n+1))-n)},~~l\in\mathbb{C},$$
is the Wright function.

\noindent Define  $U_{ad}=\{u(\cdot,x):[0,1]\rightarrow \mathcal{U} ~\mbox{ is}~ \Upsilon_t\mbox{-adapted and measurable stocahstic processes}\\\mbox{ and} \|u\|_{L^2_{\Upsilon}}\leq \eta, ~\eta>0\}.$ Let $z(t)(x)=z(t,x)$ and $$g_1\big(t,z(t)\big)(x)=\dfrac{e^{-3t}z(t,x)}{2+|z(t,x)|},\quad g_2\big(t,z(t)\big)(x)=\dfrac{e^{\pi}z(t,x)}{e^{4t}},\quad Eu(t)(x)=u(t,x),$$
$$\varsigma_q\big(t,z(t_q^-)\big)(x)=\frac{1}{4}\sin (t)z(0.20^-,x),\quad \varphi_q\big(t,z(t_q^-)\big)(x)= \frac{1}{3}\sin(t)z(0.20^-,x).$$ Now, the system (\ref{5.1}) can be written as in the abstract form (\ref{1.1}) and all the presumptions of Theorem~\ref{th1} and ~\ref{th4} are fulfilled. Thus, the Lagrange problem corresponding to (\ref{5.1}) has at least one optimal state-control pair.

\textbf{Example 6.2.} Consider 
\begin{eqnarray}  \label{5.2}
	\left \{ \begin{array}{lll} ^cD_{t}^{1.2}z(t,x)+^cD_{t}^{0.6}z(t,x)+ 5^cD_{t}^{0.4}z(t,x)+8^cD_{t}^{0.3}z(t,x)= z_{xx}(t,x)+u(t,x)\\+\frac{1}{3}{e^{-t}z(t,x)}+\dfrac{2+e^{-\pi t}z(t,x)}{1+|z(t,x)|}\dfrac{d\upsilon(t)}{dt},\quad t\in (0,0.40]\cup (0.80,1],\quad x \in [0,\pi],
		&\\z(t,x)=\cos(t)z(0.40^-,x), \quad t \in (0.40,0.80],\quad x \in [0,\pi],
		&\\z'(t,x)=\frac{1}{9}\cos(t)z(0.40^-,x), \quad t \in (0.40,0.80],\quad x \in [0,\pi],
		&\\z(t,0)=z(t,\pi)=0, &\\z(0,x)=0, \quad \dfrac{\partial z(t,x)}{\partial t}\Big\rvert_{t=0}=x+2 , \quad t \in [0,\pi],
	\end{array}\right.
\end{eqnarray}
with cost functional as
\begin{eqnarray*}
	\mathcal{I}(z,u)=\int_{0}^{1}\int_{0}^{\pi}\|z^u(t,x)\|^2dxdt+\int_{0}^{1}\int_{0}^{\pi}\|u(t,x)\|^2dxdt.
\end{eqnarray*} 
Following and collecting the discussion of the above example, we can easily achieve the result.

\section{ \textbf{Conclusion}} The present paper studied the existence and optimal control results of a multi-term time-fractional stochastic system with non-instantaneous impulses. We proved the existence of mild solutions by using Krasnoselskii's fixed point theorem. Using the minimizing sequence concept, we obtained that, under some natural assumptions, the problem admits at least one optimal pair of state-control. Finally, the obtained results have been demonstrated by the examples. For further work, one may apply an iterative technique to establish the results of the above-considered problem.

\textbf{{Acknowledgement}}
The first and third authors acknowledge UGC, India, for providing financial support through MANF F.82-27/2019 (SA-III)/ 4453 and F.82-27/2019 (SA-III)/191620066959, respectively.

\textbf{{Declarations of interest:}} None.

\textbf{ Statement of Contributions}\\
\textbf{Asma Afreen:} Conceptualization, Software, Validation, Original draft preparation, Writing- Reviewing and Editing, Visualization. \\
\textbf{Abdur Raheem:} Conceptualization, Methodology, Validation, Investigation, Reviewing, Supervision. \\
\textbf{Areefa Khatoon:} Software, Writing- Reviewing and Editing, Visualization.

\end{document}